\documentclass[12pt]{article} 
\usepackage{amsmath, amssymb}
\usepackage{amsthm} 
\usepackage{mathtools}
\usepackage{mathrsfs}
\usepackage{enumerate}
\usepackage{empheq}
\usepackage{esint}
\usepackage{bbm}
\usepackage{color}
\usepackage[T1]{fontenc}
\usepackage{textcomp}
\usepackage[abbrev]{amsrefs}
\usepackage[UKenglish]{babel}
\usepackage[UKenglish]{isodate}
\numberwithin{equation}{section}
\theoremstyle{plain} 
\newtheorem{theorem}{Theorem}[section]
\newtheorem{lemma}[theorem]{Lemma}
\newtheorem{corollary}[theorem]{Corollary}
\newtheorem{proposition}[theorem]{Proposition}

\theoremstyle{definition} 
\newtheorem{definition}[theorem]{Definition}
\newtheorem{remark}[theorem]{Remark}

\def\N{\mathbb N}

\def\R{\mathbb R}

\def\D{\mathcal D}

\def\F{\mathcal F}

\def\Re{\mathcal R}

\newcommand{\iO}{\int_\Omega}
\newcommand{\iB}{\int_{\Omega}}

\newcommand{\uv}{u_\varphi}

\usepackage{bbm}

\allowdisplaybreaks

\makeatletter
\def\address#1#2{\begingroup
\noindent\parbox[t]{7.8cm}{%
\small{\scshape\ignorespaces#1}\par\vskip1ex
\noindent\small{\itshape E-mail address}%
\/: #2\par\vskip4ex}\hfill%
\endgroup}%
\makeatother
\title{\uppercase{Concentration phenomena to a chemotaxis system with indirect signal production}} 
\author{
\bigskip \\
\textsc{Yuri Soga} 
}
\cleanlookdateon
\date{\today} 
\begin{document}

\maketitle

\begin{abstract}
We consider a parabolic-ODE-parabolic chemotaxis system with radially symmetric initial data
in a two-dimensional disk under the $0$-Neumann boundary condition. Although our system shares similar mathematical structures as the Keller--Segel system,
the remarkable characteristic of the system we consider is that its solutions cannot blow up in finite time. 
In this paper, focusing on  blow-up  solutions in infinite time, we confirm concentration phenomena at the origin. It is shown that the radially symmetric solutions of our system have a singularity like a Dirac delta function in infinite time. This means that there exist a time sequence $\{t_k\}$, a weight $m \ge 8\pi$, and a nonnegative function $f \in L^1(\Omega)$ such that
\begin{align*}
u(\cdot,t_k) \stackrel{*}{\rightharpoonup} m \delta (0) + f\ \mathrm{as}\ t_k \to \infty.
\end{align*}
We highlight this result is obtained by showing uniform-in-time boundedness of some energy functional. Moreover, we 
study whether $m = 8\pi$ or $m > 8\pi$, which is an open problem in the Keller--Segel system.
It is proved that the weight $m$ of a delta function singularity is larger than $8\pi$
under a specific assumption associated with a Lyapunov functional. This finding suggests the relationship between solutions blowing up in infinite time and an unboundedness of a Lyapunov functional, which contrasts with the Keller--Segel system.

\end{abstract}

\section{Introduction}
We consider the initial boundary value problem:
\begin{equation}\label{p}\tag{P}
\begin{cases}
 u_t = \Delta u - \nabla \cdot (u \nabla w)\qquad &\mathrm{in}\ \Omega \times (0,\infty),\\
 v_t = -v + u\qquad &\mathrm{in}\ \Omega \times (0,\infty),\\
 w_t = \Delta w -w + v \qquad &\mathrm{in}\ \Omega \times (0,\infty),\\
\dfrac{\partial u}{\partial \nu} = \dfrac{\partial w}{\partial \nu} = 0 \qquad &\mathrm{on}\ \partial \Omega \times (0,\infty),\\
(u(\cdot,0),v(\cdot, 0), w(\cdot, 0)) = (u_0(\cdot),v_0(\cdot),w_0(\cdot))\qquad &\mathrm{in}\ \Omega,
\end{cases}
\end{equation}
where $\Omega$ is a disk in two-dimensional space and $(u_0, v_0, w_0)$ satisfies
\begin{equation}\label{c}
\begin{cases}
u_0 \in C(\overline{\Omega}),\quad u_0 \ge 0\ (u_0 \not\equiv 0)\quad &\mathrm{in}\ \Omega,\\
v_0 \in  C^1(\overline{\Omega}),\quad v_0 \ge 0 &\mathrm{in}\ \Omega,\\
w_0 \in C^2(\overline{\Omega}),\quad w_0 \ge 0 &\mathrm{in}\ \Omega.
\end{cases}
\end{equation}
Our purpose in this paper is to show the concentration phenomenon of the solutions to the above parabolic-ODE-parabolic chemotaxis system.

In a wide range of biological systems, pattern formation can be observed. Mathematical mechanisms of pattern formation in biological systems play an important role in biology.
Such phenomena in ecosystems can also be described by reaction-diffusion systems with an advection term, which is so-called chemotaxis systems. The system \eqref{p} introduced by \cite{STP2013, WP1998} models cluster attack of the Mountain Pine
Beetle (MPB) driven by chemotaxis.
In the system \eqref{p}, the unknown functions $u, v$, and $w$ respectively denote the density of the flying MPB, the density of the nesting MPB, and the concentration of beetle pheromone. The movement of the flying MPB is determined by the effects of diffusion and chemotaxis. An interesting aspect of MPB's life cycle is that the flying MPB spreads randomly while moves in the direction of concentration gradient of pheromone produced by the nesting beetles. Such characteristics of \eqref{p} are different from those of the Keller--Segel system, pioneered by Keller and Segel \cite{KS1970}, where the chemical substances is directly secreted by the cells itself. 

The Keller–Segel system given by
\begin{equation}\label{KS}\tag{KS}
\begin{cases}
 u_t = \Delta u - \nabla \cdot (u \nabla w)\qquad &\mathrm{in}\ \Omega \times (0,\infty),\\
\tau w_t = \Delta w -w + u \qquad &\mathrm{in}\ \Omega \times (0,\infty)
\end{cases}
\end{equation}
has been extensively studied, specifically for the case of two-dimensional space. 
The remarkable point of the study for large-time behaviors of solutions to the Keller--Segel system is the mass critical phenomena, which are often called ``$8\pi$-problem''. To state it precisely, 
if $\|u_0\|_{L^1(\Omega)} < 8\pi$, the corresponding radially symmetric classical solution exists globally in time and remains bounded (see \cite{B1998, NSY1997}).
While if 
$\|u_0\|_{L^1(\Omega)} > 8\pi$, then a finite time blowup may occur (see \cite{N1995, M2020_1, M2020_2}). Furthermore, the model of $\tau = 0$ in \eqref{KS} proposed by Nagai \cite{N1995} has a blowup criterion using the second moment given by \cite{N1995, SS2001_2}.
Here, the threshold value $8\pi$ is deeply related to the properties of a Lyapunov functional  and some functional inequalities.
It is shown that the Keller--Segel system has the following Lyapunov functional, which is an important mathematical structure: 
\begin{align*}
\dfrac{d}{dt}\mathcal{E}(u,w)(t) + \iB u|\nabla (\log u -w)|^2 dx + \|w_t\|_{L^2(\Omega)}^2 = 0,
\end{align*}
where 
\begin{align}
\mathcal{E}(u,w)(t) := \iB u\log u dx -\iB uw dx + \dfrac{1}{2}(\|w\|_{L^2(\Omega)}^2 + \|\nabla w\|_{L^2(\Omega)}^2).\label{eq:i1}
\end{align}
The combination of the Lyapunov functional and a consequence
derived from the Trudinger--Moser inequality (see \cite{CY1988, MJ1970}) leads to the results in \cite{B1998, NSY1997}. 
We may refer the reader to \cite{H2003, BBTW2015} which are a summary of research on the Keller--Segel system and some chemotaxis systems. 

On the other hand, Lauren\c cot \cite{L2019} showed the system \eqref{p} has the similar Lyapunov functional denoted by
\begin{align}
\F(u,w)(t) := \iB u\log u dx -\iB uw dx + \dfrac{1}{2}(\|w\|_{L^2(\Omega)}^2 + \|\nabla w\|_{L^2(\Omega)}^2 + \|w_t\|_{L^2(\Omega)}^2)\label{eq:i2}
\end{align}
and nonetheless exhibits the following mass critical phenomena different from the Keller--Segel system:
\begin{itemize}
  \item The classical solution of \eqref{p} always exists globally in time regardless of the size of the initial mass $\|u_0\|_{L^1(\Omega)}$ and moreover if $\|u_0\|_{L^1(\Omega)} < 4\pi$ (or $8\pi$ in the radially symmetric case), the corresponding solution of \eqref{p} is uniformly bounded in time.
  \item there exists some initial data with $\|u_0\|_{L^1(\Omega)}\in (4\pi, \infty) \setminus 4\pi \N$ (or $\|u_0\|_{L^1(\Omega)} > 8\pi$ in the radially symmetric case) such that the corresponding solution of \eqref{p} blows up in infinite time.
\end{itemize}

It appears to be interesting that the blowup time of solutions of \eqref{p} is delayed compared to the solutions of the Keller--Segel system. Nevertheless the research on \eqref{p} seems to be confined to \cite{L2019, LS2021}. Thus we now study the solution of \eqref{p} blowing up in infinite time in order to analyze a blowup behavior of the solution to \eqref{p} in greater detail. 

The solutions blowing up in finite time of \eqref{KS} have a singularity like a Dirac delta function in the blowup procedure, which is initially shown by Herrero and Vel\'azquez \cite{HV1996, HV1997} using the asymptotic expansions methods. We call such phenomena ``chemotactic collapse''. Nagai, Senba, and Suzuki \cite{NSS2000} and Senba and Suzuki \cite{SS2001} showed chemotactic collapse occurs if the
solution of \eqref{KS} blows up in finite time $T_{\mathrm{max}} < \infty$, with its proof based on an energy estimate. Namely, the component $u$ satisfies 
\begin{align*}
u(\cdot, t) \stackrel{*}{\rightharpoonup} m \delta(0) + f\ \mathrm{in}\ \mathscr{M}(\overline{\Omega})\ \mathrm{as}\ t \to T_{\mathrm{max}} < \infty,
\end{align*}
where the weight $m$ is some constant ($m \ge 8\pi$), and $\mathscr{M}(\overline{\Omega})$ is denoted by the dual space of $C(\overline{\Omega})$ and is called the space of Radon measures on $\overline{\Omega}$. Furthermore, Suzuki \cite{ST2005} studied the model of $\tau =0$ in \eqref{KS} and proved that $m = 8\pi$, referred to as the mass quantization. A similar result is shown even in the case of infinite time in the same model by means of a blowup criterion derived from the second moment (see \cite{SS2002, ST2005, SS2001_2}). Meanwhile, this method cannot be applied to the fully parabolic system \eqref{KS}, so that whether the mass quantization ($m = 8\pi$) occurs or not ($m > 8\pi$) remains an open problem.

The motivation of this paper is to clarify behaviors of the solutions blowing up in infinite time in the system \eqref{p}, which has similar mathematical structures as the Keller--Segel system. We expect the possibility of $m > 8\pi$ in the above limit under certain initial conditions. Our conjecture arises from \cite{ML2024, TW2017} in the study of parabolic-elliptic-ODE system
\begin{equation}\label{YL}
\begin{cases}
 u_t = \Delta u - \nabla \cdot (u \nabla v)\qquad &\mathrm{in}\ \Omega \times (0,\infty),\\
0 = \Delta v - \mu + w\qquad &\mathrm{in}\ \Omega \times (0,\infty),\\
w_t = -w + u \qquad &\mathrm{in}\ \Omega \times (0,\infty),
\end{cases}
\end{equation}
where $\mu = \frac{1}{|\Omega|}\iB w(x,t)dx$. 
 The pioneer work due to Tao and Winkler \cite{TW2017} provided a novel type of mass critical phenomenon, which differs from that of the Keller--Segel system. To be more specific, Tao and Winkler \cite{TW2017} exhibited solutions blowing up in infinite time of \eqref{YL}. Subsequently in \cite{ML2024}, it was established that given specific initial conditions with $\|u_0\|_{L^1(\Omega)} \ge 32\pi$, then the corresponding radially symmetric solution of \eqref{YL} blows up in infinite time and satisfies
\begin{align*}
u(\cdot, t) \stackrel{*}{\rightharpoonup} \|u_0\|_{L^1(\Omega)} \delta(0)\ \mathrm{in}\ \mathscr{M}(\overline{\Omega})\ \mathrm{as}\ t \to \infty. 
\end{align*}
The method of the proof in \cite{ML2024} is a comparison principle.
The advantage of the special form of the second equation in the system \eqref{YL} makes it possible to transform the system of equations into a single parabolic equation with a nonlocal term. Additionally, the single parabolic equation enables approaches employing a comparison principle (see \cite[Section 4]{TW2017}).
Hence the above results in \cite{ML2024} are gained by constructing subsolutions and supersolutions for the single differential equation.
On the other hand, this approach in \cite{ML2024} cannot be available for \eqref{p} and \eqref{KS}, which are essentially the system of coupled equations.
However, for the reason that the system of both \eqref{p} and \eqref{YL} have the ODE, 
the results in \cite{ML2024, TW2017} give our conjecture of $m > 8\pi$ for the system \eqref{p}. In this paper, we show a concentration of the mass in infinite time by invoking the Lyapunov functional of \eqref{p}. Let us introduce the definition. 

\begin{definition}
Let $(u,v,w)$ be a classical solution to \eqref{p} in $\Omega \times [0, \infty)$. 
\begin{enumerate}
\item[\rm{(i)}] We shall call $(u,v,w)$ a grow-up solution if it holds that
\begin{align*}
\limsup_{t \to \infty}\|u(t)\|_{L^\infty(\Omega)} = \infty.
\end{align*}
\item[\rm{(ii)}] We shall call $x_0 \in \overline{\Omega}$ a grow-up point if there exist $\{t_k\} \subset (0,\infty)$ and $\{x_k\} \subset \overline{\Omega}$ such that
\begin{align*}
t_k \to \infty,\quad x_k \to x_0,\quad u(x_k, t_k) \to \infty\ \mathrm{as}\ k \to \infty.
\end{align*}
Moreover, we shall denote by $\mathcal{G}$ the set of grow-up points.
\end{enumerate}
\end{definition}
We are interested in studying behaviors of the grow-up solution to \eqref{p} near grow-up points.
Therefore, when considering the grow-up solution of \eqref{p}, it is essential for grow-up points to exist (see Lemma \ref{lem:3.1111}).

\medskip
\textbf{Main Results.} In this paper, we always assume the radially symmetric grow-up solution of \eqref{p}. Hence our first task is to show that the grow-up point is the only origin.
\begin{theorem}\label{th:2}
Let $\Omega = B_R \subset \R^2$ with some $R > 0$ and $(u_0,v_0,w_0)$ be a triplet of radially symmetric functions satisfying \eqref{c}. Then $\mathcal{G} = \{0\}$.
\end{theorem}

We next state the mass concentration phenomenon to the grow-up solution of \eqref{p} at the origin.

\begin{theorem}\label{th:1}
Let $\Omega = B_R \subset \R^2$ with some $R > 0$ and $(u_0,v_0,w_0)$ be a triplet of radially symmetric functions satisfying \eqref{c}. 
If $(u,v,w)$ is a grow-up solution of \eqref{p}, then there exist a sequence $\{t_k\}$ of time with $t_k \to \infty$ as $k \to \infty$ and $m \ge 8\pi$ as well as a nonnegative function $f \in 
L^1(\Omega) \cap C(\overline{\Omega} \setminus \{0\})$ such that
\begin{align*}
u(\cdot, t_k) \stackrel{*}{\rightharpoonup} m \delta(0) + f\ \mathrm{in}\ \mathscr{M}(\overline{\Omega})\  \mathrm{as}\ k \to \infty.
\end{align*}
That is to say, it holds that for any $\xi \in C(\overline{\Omega})$,
\begin{align*}
\iB u(x,t_k)\xi(x) dx \rightarrow m\xi(0) + \iB f(x)\xi(x) dx\ \mathrm{as}\ k \to \infty.
\end{align*}
\end{theorem}

\begin{remark}
Theorem \ref{th:1} is a similar result as \cite{NSS2000}, which is the study for the solutions blowing up in finite time of the Keller--Segel system. However, we emphasize that our theorem is not a direct consequence of the argument in \cite{NSS2000}. In the proof of the equivalence between the unboundedness of $L\log L$ for the first component $u$ around $x_0$ and $x_0$ is the blowup points, the method in \cite{NSS2000} is to obtain the local estimates in time for the solutions of the Keller--Segel system. On the other hand, our approach is based on the global estimates in time that consistently accounts for the global-in-time solution of \eqref{p}. Moreover, whether the time $t$ is a full sequence or a subsequence in the blowup procedure also differs. In fact, $f$ is denoted by the limit function of $u$ in the area away from the origin. Then, in \cite{NSS2000}, $\lim_{t \to T_\mathrm{max}}u(x,t)$ exists in the domain excluding the neighborhood of the origin, which is derived from the boundedness of $u_t$ in such an area.
Meanwhile, it dose not hold in our case as $t \to \infty$ due to the possibility of oscillation, thereby we show Theorem \ref{th:1} with the use of the compact theorem of the H\"older space. Consequently, we obtain the subsequence of time satisfying the convergence in Theorem \ref{th:1}.

\end{remark}

The next challenge is whether $m = 8\pi$ or $m > 8\pi$ in Theorem \ref{th:1}. We will present the following theorem for this question.

\begin{theorem}\label{th:3}
Let $\Omega = B_R \subset \R^2$ with some $R > 0$ and $(u_0,v_0,w_0)$ be a triplet of radially symmetric functions satisfying \eqref{c} and $\|u_0\|_{L^1(\Omega)} > 8\pi$. If a grow-up solution $(u,v,w)$ of \eqref{p} satisfies 
\begin{align}
\inf_{t \ge 0} \F (t) > -\infty,\label{eq:i3}
\end{align}
then there exists a sequence $\{t_k\}$ of time with $t_k \to \infty$ as $k \to \infty$ such that
\begin{align*}
u(\cdot, t_k) \stackrel{*}{\rightharpoonup} \|u_0\|_{L^1(\Omega)} \delta(0)\ \mathrm{in}\ \mathscr{M}(\overline{\Omega})\ \mathrm{as}\ k \to \infty.
\end{align*}
\end{theorem}

Theorem \ref{th:3} contributes to deriving the following result, which can never be attained by the Keller--Segel system.
\begin{corollary}\label{cor:1}
Let $\Omega = B_R \subset \R^2$ with some $R > 0$ and $(u_0,v_0,w_0)$ be a triplet of radially symmetric functions satisfying \eqref{c} and $\|u_0\|_{L^1(\Omega)} > 8\pi$. Then a radially symmetric grow-up solution of \eqref{p} satisfies either of the followings:
\begin{enumerate}
\item[\rm{(i)}] there exists a sequence $\{t_k\}$ of time with $t_k \to \infty$ as $k \to \infty$ such that
\begin{align*}
u(\cdot, t_k) \stackrel{*}{\rightharpoonup} \|u_0\|_{L^1(\Omega)} \delta(0)\ \mathrm{in}\ \mathscr{M}(\overline{\Omega})\ \mathrm{as}\ k \to \infty.
\end{align*}
\item[\rm{(ii)}] It holds that
\begin{align*}
\inf_{t \ge 0} \F(t) = -\infty.
\end{align*}
\end{enumerate}

\end{corollary}

Theorem \ref{th:3} implies that $m = \|u_0\|_{L^1(\Omega)} > 8\pi$ in Theorem \ref{th:1} holds under the assumption \eqref{eq:i3}.
Moreover, Theorem \ref{th:3} is unachievable in the Keller--Segel system. Indeed, Mizoguchi \cite{M2020_1, M2020_2} in the study of the Keller--Segel system constructed the ordinary differential inequality for the Lyapunov functional \eqref{eq:i1} in the form of
\begin{align*}
-\dfrac{d}{dt}\mathcal{E}(u,w)(t) \ge C_1\left(-\mathcal{E}(u,w)(t)\right)^\theta -C_2,
\end{align*}
where $(u,w)$ is a radially symmetric solution of the Keller--Segel system and $\theta > 1$. This inequality implies that if a radially symmetric solution $(u,w)$ exists globally in time, then the Lyapunov functional $\mathcal{E}$ is uniformly bounded from below in time $(0,\infty)$. Combining the fact and the investigation of a stationary problem in the domain except the blowup point, Mizoguchi \cite{M2020_1} extended the results of \cite{HW2001} in the radial case as follows:
\begin{itemize}
  \item If the radially symmetric classical solution of \eqref{KS} corresponding to $\|u_0\|_{L^1(\Omega)} > 8\pi$ exists globally in time, then the solution is always uniformly bounded in time.
  \item there exists some initial data with $\|u_0\|_{L^1(\Omega)} > 8\pi$ such that the corresponding radially symmetric solution of \eqref{KS} blows up in finite time.
\end{itemize}
The important point to note here is the first result above, which says a concentration of the mass in infinite time cannot happen in the supercritical case ($\|u_0\|_{L^1(\Omega)}  > 8\pi$). That is to say, the assumption in Theorem \ref{th:3} is false in the case of the Keller--Segel system.
We draw inspiration from the technique in \cite{M2020_1}, thereby we show the concentration phenomena of the solution of \eqref{p} under the assumption \eqref{eq:i3}.
Here, as mentioned above, the radially symmetric global solution of the Keller--Segel system always satisfies the boundedness from below in time $(0, \infty)$ for the Lyapunov functional $\mathcal{E}$. This fact was revealed by Mizoguchi \cite{M2020_1}, but there is a background conjectured by Ohtsuka, Senba, and Suzuki \cite{OSS2007} regarding the model of $\tau = 0$ in \eqref{KS}. In contrast, whether the assumption \eqref{eq:i3} holds for a grow-up solution of \eqref{p} is an open problem, yet Theorem \ref{th:3} gives the relationship between grow-up solution and the boundedness of the Lyapunov functional of \eqref{p}.

\medskip
\begin{remark} Theorem \ref{th:3} enables us to suggest three arguments as follows for a grow-up behavior in the system \eqref{p}. 
\begin{itemize}
\item The first one is about the value of the weight $m$ of the delta function. We now obtain $m = \|u_0\|_{L^1(\Omega)} > 8\pi$ in Theorem \ref{th:3}, which differs from \cite{NSS2000} and is not observed in \cite{SS2002, ST2005}. Also, we succeed in extending the range of the weight $m = \|u_0\|_{L^1(\Omega)}$ compared to \cite{ML2024}, in which the range is $m \ge 32\pi$.
\item The second one is about the global behavior of the Lyapunov functional. If it holds that the weight $m$ should be $8\pi$ in Theorem \ref{th:1} just as \cite{SS2002, ST2005}, Theorem \ref{th:3} indicates there exists a radially symmetric grow-up solution of \eqref{p} corresponding the initial data $\|u_0\|_{L^1(\Omega)} > 8\pi$ such that $\inf_{t \ge 0} \F(t) = -\infty$. This argument is unlike for the Keller--Segel system. 
\item The third one is about the possibility that the assumption in Theorem \ref{th:3} is false, that is, the radially symmetric global in time solutions of \eqref{p} satisfying \eqref{eq:i3} remain bounded. However, we already know that there exist radially symmetric grow-up solutions of \eqref{p} with the initial data $\|u_0\|_{L^1(\Omega)} > 8\pi$ (see \cite{L2019}). Therefore, if the assumption is false, we can conclude that the radially symmetric grow-up solutions of \eqref{p} satisfy $\inf_{t \ge 0} \F(t) = -\infty$. This gives rise to the conclusion different from the Keller--Segel system.
\end{itemize}
In any situation, we reveal new differences about large-time behaviors of solutions to \eqref{p} and the Keller--Segel system thanks to Theorem \ref{th:3}.
\end{remark}

\medskip
\textbf{Plan of this paper.}
This paper is organized as follows. In Section 2, we prepare some estimates and properties which the solution of \eqref{p} satisfies.
After that, we will introduce fundamental inequalities and lemmas to proceed an argument in a localized area in Subsection \ref{subsec:3.1}.
In Subsection \ref{subsec:3.2}, we will derive time-independent estimates of $L\log L$ for the first component $u$  in a localized area. Subsection \ref{subsec:3.3} is devoted to the Alikakos--Moser iteration argument. We will obtain time-independent a priori estimates of the solution in Subsection \ref{subsec:3.3}. Subsequently, we will enter the stage of proving our theorem.
Subsection \ref{subsec:3.4} is devoted to the proof of Theorem \ref{th:2}, which follows from an $\varepsilon$-regularity theorem.
We finally will prove Theorem \ref{th:1} (Subsection \ref{subsec:3.5}) and Theorem \ref{th:3} (Subsection \ref{subsec:3.6}).

\section{Preliminaries}

In this section, we collect several properties of solutions to \eqref{p}.
Through this paper,  we take advantage of the well known properties regarding the Laplacian with $0$-Neumann boundary condition in $L^p(\Omega)$ for $1 < p < \infty$, where $\Omega$ is a bounded domain with a smooth boundary in two-dimensional space. Let denote $\Delta_{N}$ the realization of the Laplacian in $L^p(\Omega)$ for $1<p < \infty$ defined on
\begin{align*}
D(\Delta_N) := \left\{z \in W^{2,p}(\Omega) ; \dfrac{\partial z}{\partial \nu} = 0\ \mathrm{on}\ \partial \Omega\right\}.
\end{align*}
Then $\Delta_N$ is sectorial operator so that $\Delta_N$ is a generator of an analytic semigroup. We call such an analytic semigroup denoted by $\{e^{t\Delta_N}\}_{t \ge 0}$ the Neumann heat semigroup. 
For more details we refer the reader to \cite{HD1981, LA2015}.

We now emphasize that the unique classical solution to \eqref{p} exists globally in time regardless of the initial mass (see \cite[Theorem 1.1]{L2019}).

\begin{proposition}\label{pro:loc1}
Let $(u_0,v_0,w_0)$ be as in \eqref{c}. Then there exists a unique positive classical solution $(u, v, w)$ satisfying 
\begin{align*}
&u \in C^{2,1}(\overline{\Omega} \times (0,\infty)) \cap C([0,\infty) ; C(\overline{\Omega})),\\
&v \in C^{1,1}(\overline{\Omega}\times (0,\infty)) \cap C^1([0,\infty); C(\overline{\Omega})),
\\
&w \in C^{2,1}(\overline{\Omega} \times (0,\infty)) \cap C([0,\infty) ; C^1(\overline{\Omega})).
\end{align*}
\end{proposition}
Here, the basic idea of the proof of local existence of solutions is to make use of Banach's fixed point theorem (see \cite{QS2007, W1980}) and the Schauder estimate (see \cite{LSU1968}).

We next provide basic properties and useful inequalities to the solution of \eqref{p}.

\begin{proposition}\label{prop:3.1.0}
Let $(u,v,w)$ be a solution to \eqref{p} in $\Omega \times [0,\infty)$. Then the followings hold that:
\begin{enumerate}
\item[\rm{(i)}] It holds that for all $t \in (0,\infty)$, 
\begin{align}
\|u(t)\|_{L^1(\Omega)} = \|u_0\|_{L^1(\Omega)},\label{pro:3.1.1}
\end{align}
which is the mass conservation law.
\item[\rm{(ii)}] It holds that for all $t \in (0,\infty)$,
\begin{align}
\|v(t)\|_{L^1(\Omega)} = e^{-t} \|v_0\|_{L^1(\Omega)} + (1-e^{-t})\|u_0\|_{L^1(\Omega)}\label{pro:3.1.2}
\end{align}
and directly for all $t \in (0,\infty)$,
\begin{align*}
\|v(t)\|_{L^1(\Omega)} \leq \max\{\|u_0\|_{L^1(\Omega)}, \|v_0\|_{L^1(\Omega)}\}.
\end{align*}
\item[\rm{(iii)}] It holds that for all $t \in (0,\infty)$, 
\begin{align}
\|w(t)\|_{L^1(\Omega)} \leq \max\{\|u_0\|_{L^1(\Omega)}, \|v_0\|_{L^1(\Omega)}, \|w_0\|_{L^1(\Omega)}\}.\label{pro:3.1.4}
\end{align}
\end{enumerate}
\end{proposition}

The next proposition is a fairly straightforward of semigroup properties.
\begin{proposition}
Suppose that $(u, v, w)$ be a solution to \eqref{p} in $\Omega \times [0, \infty)$ with \eqref{c}.
Then it holds that:
\begin{enumerate}
\item[\rm{(a)}] Let $p \in (1,\infty)$, then there exists $C_1 > 0$ such that for all $t \in (0,\infty)$,
\begin{align}\label{pro:3.2.1}
\|w(t)\|_{L^p(\Omega)} \leq C_1(\max\{\|u_0\|_{L^1(\Omega)}, \|v_0\|_{L^1(\Omega)}\} + \|w_0\|_{L^p(\Omega)}).
\end{align}
\item[\rm{(b)}] Let $p \in (1,2)$, then there exists $C_2 > 0$ such that for all $t \in (0,\infty)$,
\begin{align}\label{pro:3.2.2}
\|\nabla w(t)\|_{L^p(\Omega)} \leq C_2(\max\{\|u_0\|_{L^1(\Omega)}, \|v_0\|_{L^1(\Omega)}\} + \|\nabla w_0\|_{L^p(\Omega)}).
\end{align}
\end{enumerate}
\end{proposition}

In addition to the above proposition, we obtain a pointwise estimate for $w$ in the case $(u,v,w)$ is a radially symmetric solution of \eqref{p}.

\begin{proposition}[{\cite[Lemma 3.2]{WM2013}}]\label{pro:322}
Let $(u, v, w)$ be a radially symmetric solution to \eqref{p} in $\Omega \times [0, \infty)$ with \eqref{c}. Then, for $p \in (1,2)$, it holds that there exists a positive constant $C$ depending on $p$ such that for all $(x,t) \in \Omega \times (0,\infty)$,
\begin{align*}
w(x,t) \leq C(p) (\max\{\|u_0\|_{L^1(\Omega)}, \|v_0\|_{L^1(\Omega)}\} + \|w_0\|_{W^{1,p}(\Omega)})|x|^{-\frac{2-p}{p}}.
\end{align*}

\end{proposition}

We recall the following Lyapunov functional which was established in \cite[Lemma 3.2]{L2019}.
\begin{lemma}\label{lem:3.41}
Let $(u,v,w)$ be a solution to \eqref{p} in $\Omega \times [0, \infty)$. Then it holds that for all $t \in (0,\infty)$,
\begin{align}\label{eq:2.1}
\dfrac{d}{dt}\F(t) + \D(t) = 0,
\end{align}
where
\begin{align*}
\F(t) &= \iO u\log u dx - \iO uw dx + \dfrac{1}{2}\left(\|w_t(t)\|_{L^2(\Omega)}^2 + \|\nabla w(t)\|_{L^2(\Omega)}^2 + \|w(t)\|_{L^2(\Omega)}^2\right),\\
\D(t) &= \iO u|\nabla (\log u - w)|^2 dx + 2\|w_t(t)\|_{L^2(\Omega)}^2 + \|\nabla w_t(t)\|_{L^2(\Omega)}^2.
\end{align*}
\end{lemma}

\section{Concentration around grow-up points}
In the case of the Keller--Segel system, it is well known that if we can obtain the uniform boundedness in time of $\iB u\log u dx$, the global existence and boundedness of solutions can be established by Alikakos--Moser's iteration scheme. This section is devoted to the study of the argument stated above in a localized area. Throughout the proof, $B_r(x_0)$ stands for the set $\{x \in \R^2 ; |x - x_0| < r \}$ for any $r > 0$ and $x_0 \in \R^2$.

\subsection{Fundamental inequalities and lemmas}\label{subsec:3.1}
We first prove that there exist grow-up points when considering grow-up solutions of \eqref{p}.
\begin{lemma}\label{lem:3.1111}
Let $(u,v,w)$ be a grow-up solution of \eqref{p} in $\Omega \times [0,\infty)$. Then, it holds that $\mathcal{G} \not= \emptyset$.
\end{lemma}

\begin{proof}
The definition of $\limsup$ and a grow-up solution implies that there exists a subsequence of time $\{t_k\}$ with $t_k \to \infty$ as $k \to \infty$ such that 
\begin{align*}
\lim_{k \to \infty} \|u(t_k)\|_{L^\infty(\Omega)} = \infty.
\end{align*}
Moreover, due to Proposition \ref{pro:loc1}, we can take $\{x_k\} \subset \overline{\Omega}$ such that
\begin{align*}
u(x_k,t_k) = \|u(t_k)\|_{L^\infty(\Omega)}.
\end{align*}
Since we assume $\Omega$ is a bounded domain in $\R^2$, the Bolzano--Weierstrass theorem allows us to show that there exist a subsequence of $x_k$, 
which is denoted by $x_k$, and $x_0 \in \overline{\Omega}$ such that $x_k \to x_0$ as $k \to \infty$. This implies the existence of grow-up points.
\end{proof}

We introduce a cut-off function $\varphi$ to obtain localized estimates around the grow-up points. The following lemma agrees with the one given in \cite{SS2001} and \cite[Lemma 2.3]{FS2016}.
\begin{lemma}\label{lem:30}
Let $x_0 \in \overline{\Omega}, n \in \N$, and $r > 0$. Then there exists a function $\varphi = \varphi_{(x_0,r,n)} \in C^\infty_c(\R^2)$ satisfying
\begin{align*}
  &\varphi(x)=
  \begin{cases}
    1\quad \mathrm{in}\ B_r(x_0), \\
    0\quad \mathrm{in}\ \R^2 \setminus B_{2r}(x_0),
  \end{cases}\\
  &0 \leq \varphi \leq 1\ \mathrm{in}\ \R^2,\\
  &\dfrac{\partial \varphi}{\partial \nu} = 0\ \mathrm{on}\ \partial \Omega,\\
  &|\nabla \varphi| \leq A\varphi^{1-\frac{1}{n}},\ |\Delta \varphi| \leq B\varphi^{1-\frac{2}{n}}\ \mathrm{in}\ \R^2,
\end{align*}
where $A, B$ are positive constants determined by $n, r$.
\end{lemma}

We often use the fundamental inequalities which were established in \cite{BHN1994, NSY1997} to derive $L^p$-estimates for a solution $(u,v,w)$ of \eqref{p}: let $\varepsilon > 0$, there is a positive constant $C(\varepsilon) > 0$ depending only on $\varepsilon$ and $\Omega$ such that
\begin{align*}
\|f\|_{L^3(\Omega)}^3 \leq \varepsilon \|f\|_{H^1(\Omega)}^2 \|f\log |f|\|_{L^1(\Omega)} + C(\varepsilon)\|f\|_{L^1(\Omega)}
\end{align*}
for any $f \in H^1(\Omega)$, where $\Omega$ is a bounded domain in a two-dimensional space.
The inequalities in the following lemma introduced by \cite{NSS2000} are a localized version of 
the inequality mentioned above. The proof is based on the Sobolev inequality in a bounded domain within a two-dimensional space:
\begin{align}
\|z\|_{L^2(\Omega)}^2 \leq K_{\mathrm{Sob}}^2 (\|\nabla z\|_{L^1(\Omega)}^2 + \|z\|_{L^1(\Omega)}^2)\label{sobine:1}
\end{align}
for any $z \in W^{1,1}(\Omega)$, where $K_{\mathrm{Sob}} > 0$ is a constant depending only on $\Omega$. 

\begin{lemma}[{\cite[Lemma 3.2]{NSS2000}}]\label{lem:3.3}
Let $x_0 \in \overline{\Omega}, 0 < r \ll 1$, and $\varphi = \varphi_{(x_0,r, n)}$ be as in Lemma \ref{lem:30}.
Then a positive function $z \in H^1(\Omega)$ satisfies the following inequalities:
\begin{enumerate}
\item[\rm{(i)}] It holds that 
\begin{align*}
\iB z^2 \varphi dx \leq 2K_{\mathrm{Sob}}^2 \int_{B_{2r}(x_0) \cap \Omega}z dx \iB \dfrac{|\nabla z|^2}{z}\varphi dx + K_{\mathrm{Sob}}^2 \left(\dfrac{A^2}{2} + 1\right)\|z\|_{L^1(B_{2r}(x_0) \cap \Omega)}^2.
\end{align*} 
\item[\rm{(ii)}] It holds that for any $s > 1$,
\begin{align*}
\iB z^2 \varphi dx &\leq \dfrac{4K_{\mathrm{Sob}}^2}{\log s} \int_{B_{2r}(x_0) \cap \Omega}\left(z\log z + \dfrac{1}{e}\right) dx \iB \dfrac{|\nabla z|^2}{z}\varphi dx\\
&\hspace{0.5cm} + C\|z\|_{L^1(B_{2r}(x_0)\cap \Omega)}^2 + 3s^2|\Omega|,
\end{align*}
where $C$ is a positive constant determined by $K_{\mathrm{Sob}}$ and $A$.
\item[\rm{(iii)}] It holds that for any $s > 1$,
\begin{align*}
\iB z^3 \varphi dx &\leq \dfrac{72K_{\mathrm{Sob}}^2}{\log s} \int_{B_{2r}(x_0) \cap \Omega}\left(z\log z + \dfrac{1}{e}\right) dx \iB |\nabla z|^2\varphi dx\\
&\hspace{0.5cm} + C\|z\|_{L^1(B_{2r}(x_0) \cap \Omega)}^3 + 10s^3|\Omega|,
\end{align*}
where $C$ is a positive constant determined by $K_{\mathrm{Sob}}$ and $A$.
\end{enumerate}
\end{lemma}

\subsection{$L\log L$–bound around grow–up points}\label{subsec:3.2}
We next will introduce a localized Lyapunov functional in Lemma \ref{lem:3.4}.
\begin{lemma}\label{lem:3.4}
Let $x_0 \in \overline{\Omega}$ and $\varphi = \varphi_{(x_0,r,n)}$ be as in Lemma \ref{lem:30}. Then it holds that for any $t \in (0,\infty)$,
\begin{align}
\dfrac{d}{dt}\F_{\varphi}(t) + \D_{\varphi}(t) = \dfrac{d}{dt}\int_{\Omega} u\varphi dx + \Re(u,w,\varphi),\label{eq:4.0.1}
\end{align}
where
\begin{align*}
\F_{\varphi}(t) &:= \iB (u\log u - uw)\varphi dx + \dfrac{1}{2}\iB (|w_t|^2 + |\nabla w|^2 +  w^2)\varphi dx,\\
\D_{\varphi}(t) &:= \iB u |\nabla (\log u - w)|^2\varphi dx + \iB (2|w_t|^2 + |\nabla w_t|^2) \varphi dx,\\
\Re(u,w,\varphi) &:= \iB (u\log u)\Delta \varphi dx + \dfrac{1}{2}\iB |w_t|^2 \Delta \varphi dx\\ 
&\hspace{0.5cm}+ \iB \{ (1+w)\nabla u + (u\log u -uw -w_t)\nabla w \} \cdot \nabla \varphi dx.
\end{align*}
Moreover, we call $\F_\varphi$ the localized Lyapunov functional.
\end{lemma}

\begin{proof}
We can perform the calculation in the following way using the first equation of \eqref{p} and integrating by parts:
\begin{align*}
&\iB u_t (\log u - w) \varphi dx\\
&= \iB \nabla \cdot (\nabla u - u \nabla w)(\log u -w) \varphi dx\\
&= - \iB u |\nabla (\log u - w)|^2 \varphi dx - \iB (\log u - w)(\nabla u - u\nabla w )\cdot \nabla \varphi dx.
\end{align*}
At the same time, we have
\begin{align*}
\iB u_t (\log u - w) \varphi dx = \dfrac{d}{dt}\iB (u\log u - uw)\varphi dx - \dfrac{d}{dt}\iB u\varphi dx + \iB uw_t \varphi dx.
\end{align*}
Observing $\iB uw_t\varphi dx$, from the second equation of \eqref{p} and by integrating by parts, we get
\begin{align*}
\iB u w_t \varphi dx &= \iB (v_t + v)w_t \varphi dx\\
&= \iB v_t w_t \varphi dx + \iB v w_t \varphi dx\\
&= \dfrac{1}{2}\dfrac{d}{dt}\iB |w_t|^2 \varphi dx + \iB |\nabla w_t|^2 \varphi dx + \iB w_t \nabla w_t \cdot \nabla \varphi dx + \iB |w_t|^2 \varphi dx\\
&\hspace{0.5cm}+ \iB |w_t|^2 \varphi dx + \dfrac{1}{2}\dfrac{d}{dt}\iB |\nabla w |^2 \varphi dx + \iB w_t \nabla w \cdot \nabla \varphi dx + \dfrac{1}{2}\dfrac{d}{dt}\iB w^2 \varphi dx\\
&= \dfrac{1}{2}\dfrac{d}{dt}\left(\iB |w_t|^2 \varphi dx + |\nabla w|^2\varphi dx + w^2 \varphi dx\right) + 2\iB |w_t|^2 \varphi dx\\
&\hspace{0.5cm}+ \iB |\nabla w_t|^2 \varphi dx - \dfrac{1}{2}\iB |w_t|^2 \Delta \varphi dx + \iB w_t \nabla w \cdot \nabla \varphi dx.
\end{align*}
Combining these gives \eqref{eq:4.0.1} and completes the proof. 
\end{proof}

The following lemma is intended to show Proposition \ref{prop:3.11}. It is worth pointing out that we succeed in obtaining the quasi-dissipative estimates of $\iB u \log u \varphi dx$ and $\iB |\nabla w|^2 \varphi dx$, which is different from the calculation in \cite{NSS2000}.

\begin{lemma}\label{lem:3.5}
Let $x_0 \in \overline{\Omega}, 0 < r \ll 1$, and $n$ be sufficiently large, and more let be $\varphi = \varphi_{(x_0,r,n)}$ as in Lemma \ref{lem:30}. Then for all $\varepsilon > 0$ there exists $C(\varepsilon) > 0$ such that 
\begin{align}
\dfrac{d}{dt}\left(E(t) + \dfrac{1}{2}\iB v^2 \varphi dx \right) &+ E(t) + \iB v^2 \varphi dx + \dfrac{3}{4}\iB \dfrac{|\nabla u|^2}{u}\varphi dx\notag\\
& + \dfrac{1}{16\varepsilon}\iB  |w_t|^2 \varphi dx + \dfrac{1}{4\varepsilon}\iB 
|\Delta w|^2 \varphi dx\notag\\
& \leq 10\varepsilon \iB u^2 \varphi dx + \dfrac{5}{4\varepsilon}\iB v^2 \varphi dx + C(\varepsilon),\quad t \in (0,\infty),\label{eq:3.4.1}
\end{align}
where
\begin{align*}
E(t) = \iB (u\log u)\varphi dx + \dfrac{1}{2\varepsilon}\iB |\nabla w|^2 \varphi dx + \dfrac{1}{4\varepsilon}\iB w^2 \varphi dx.
\end{align*}
\end{lemma}

\begin{proof}
The key idea of the proof is to impose differentiation on $\varphi$ by integrating by parts. 
Let us examine the term $\iB u\log u\varphi dx$. By using the first equation of \eqref{p} and integrating by parts, it holds that 
\begin{align}
\dfrac{d}{dt}\left(\iB u\log u \varphi dx\right) &= \iB u_t \log u \varphi dx + \iB u_t \varphi\notag\\
&=-\iB \dfrac{|\nabla u|^2}{u}\varphi dx - \iB \log u \nabla u \cdot \nabla \varphi dx + \iB \nabla w \cdot \nabla u \varphi dx\notag\\
&\hspace{0.5cm}+\iB u \log u \nabla w \cdot \nabla \varphi dx - \iB \nabla u \cdot \nabla \varphi dx + \iB u \nabla w \cdot \nabla \varphi dx\notag\\
&= -\iB \dfrac{|\nabla u|^2}{u}\varphi dx + \iB u(\log u +1) \nabla w \cdot \nabla \varphi dx\notag\\
&\hspace{0.5cm}+ \iB (\nabla w \cdot \nabla u) \varphi dx - \iB (1+\log u)\nabla u \cdot \nabla \varphi dx.\label{eq:3.4.2}
\end{align}
Moreover, Lemma \ref{lem:30} and the inequality:
\begin{align*}
|u\log u| \leq u\log u + \frac{2}{e}
\end{align*}
lead to that 
\begin{align}
-\iB (1 + \log u )\nabla u \cdot \nabla \varphi dx &= \iB u \Delta \varphi dx + \iB u\log u \Delta \varphi dx + \iB \nabla u \cdot \nabla \varphi dx\notag\\
&= \iB u\log u \Delta \varphi dx\notag\\
&\leq B \iB |u \log u | \varphi^{1-\frac{2}{n}} dx\notag\\
&\leq B \iB u\log u \varphi^{1-\frac{2}{n}} dx + 2B \iB \dfrac{1}{e}\varphi^{1-\frac{2}{n}} dx.\notag
\end{align}
Now, we can find a positive constant $C$ satisfying
\begin{align}
B\left(u \log u + \dfrac{2}{e}\right) \leq u^{1 + \frac{1}{n}} + C.\label{eq:3.4.2.1}
\end{align}
Hence, it holds that
\begin{align}
-\iB (1 + \log u )\nabla u \cdot \nabla \varphi dx &\leq \iB u ^{1 + \frac{1}{n}} \varphi^{1 - \frac{2}{n}} dx +  C|\Omega|\notag.
\end{align}
We remark that there exists $ n_0 \in \N$ such that
\begin{align*}
\left(1 - \dfrac{2}{n}\right)\cdot \left( \dfrac{2n}{n+1}\right) = \dfrac{2(n-2)}{n + 1} \ge 1 \quad \mathrm{for\ all}\ n \ge n_0,
\end{align*}
and $\varphi^\alpha \leq \varphi$ for $\alpha \ge 1$ due to the properties of $\varphi$. Here, we may choose sufficiently large $ n \ge n_0$. 
Using the Young inequality and the H\"{o}lder inequality, we have for any $\varepsilon > 0$
\begin{align}
-\iB (1 + \log u )\nabla u \cdot \nabla \varphi dx &\leq \left(\iB u^2 \varphi dx\right)^\frac{n+1}{2n} |\Omega|^\frac{n-1}{2n} + C\notag\\
&\leq \varepsilon \iB u^2 \varphi dx + C(\varepsilon).\label{eq:3.4.4}
\end{align}
As to the rest of the term on the right-hand side of \eqref{eq:3.4.2}, by substituting 
the third equation of \eqref{p} and integrating by parts we obtain
\begin{align*}
&\iB (\nabla w \cdot \nabla u) \varphi dx + \iB u(\log u + 1) \nabla w \cdot \nabla \varphi dx\\
&= -\iB u \nabla w \cdot \nabla \varphi dx - \iB u (\Delta w) \varphi dx + \iB u(\log u + 1) \nabla w \cdot \nabla \varphi dx\\
&=- \iB u (\Delta w) \varphi dx + \iB u \log u \nabla w \cdot \nabla \varphi dx\\
&=- \iB u (w_t + w - v)\varphi dx + \iB u\log u \nabla w \cdot \nabla \varphi dx\\
&\leq - \iB u w_t dx + \iB u v \varphi dx + \iB u\log u \nabla w \cdot \nabla \varphi dx\\
&= - \iB u w_t dx + \iB u v \varphi dx -\iB (u\log u ) w\Delta \varphi dx - \iB (1+\log u)w\nabla u \cdot \nabla \varphi dx.
\end{align*}
Similar to the discussion \eqref{eq:3.4.2.1} and owing to Lemma \ref{lem:30}, we get
\begin{align}
&\iB (\nabla w \cdot \nabla u) \varphi dx + \iB u(\log u + 1) \nabla w \cdot \nabla \varphi dx\notag\\
&\leq 4\varepsilon \iB u^2 \varphi dx + \dfrac{1}{16\varepsilon}\iB |w_t|^2 \varphi dx + \varepsilon \iB u^2 \varphi dx + \dfrac{1}{4\varepsilon}\iB v^2 \varphi dx\notag\\
&\hspace{0.5cm}+\iB u^{1+\frac{1}{n}} w \varphi^{1-\frac{2}{n}} dx + C + A\iB |(1+\log u)\nabla u|w \varphi^{1-\frac{1}{n}} dx\notag\\
&\leq 4\varepsilon \iB u^2 \varphi dx + \dfrac{1}{16\varepsilon}\iB |w_t|^2 \varphi dx + \varepsilon \iB u^2 \varphi dx + \dfrac{1}{4\varepsilon}\iB v^2 \varphi dx + C + I + II,\label{eq:3.4.5}
\end{align}
where
\begin{align*}
I &= \iB u^{1+\frac{1}{n}} w \varphi^{1-\frac{2}{n}} dx,\\
II &= A\iB |(1+\log u)\nabla u|w \varphi^{1-\frac{1}{n}} dx.
\end{align*}
Here, we have for sufficiently large $n \ge n_0$ and any $\varepsilon > 0$
\begin{align}
I &\leq \left(\iB u^2 \varphi dx\right)^{\frac{n+1}{2n}}\left(\iB w^{\frac{2n}{n-1}} dx\right)^{\frac{n-1}{2n}}\notag\\
&\leq \varepsilon \iB u^2 \varphi dx + C(\varepsilon)\label{eq:3.4.6}
\end{align}
due to \eqref{pro:3.2.1}. As to the terms $II$, by virtue of the H\"{o}lder inequality and \eqref{pro:3.2.1} we can deduce that
\begin{align}
II &= A\iB w u^{-\frac{1}{2}}|\nabla u| u^{\frac{1}{2}}|1 + \log u| \varphi^{1-\frac{1}{n}}dx\notag\\
&\leq A\left(\iB w^6 \varphi^{1-\frac{6}{n}}dx\right)^\frac{1}{6}\left(\iB \dfrac{|\nabla u|^2}{u}\varphi dx\right)^\frac{1}{2}\left(\iB u^\frac{3}{2}|1+\log u|^3 \varphi dx\right)^\frac{1}{3}\notag\\
&\leq C\left(\iB \dfrac{|\nabla u|^2}{u}\varphi dx\right)^\frac{1}{2}\left(\iB u^\frac{3}{2}|1+\log u|^3 \varphi dx\right)^\frac{1}{3}.\notag
\end{align}
Noticing that there exists $C > 0$ such that
\begin{align*}
u^\frac{3}{2} |1 + \log u|^3 \leq C(u^2 + 1),
\end{align*}
form the H\"{o}lder inequality and the Young inequality we give
\begin{align}
II &\leq C\left(\iB \dfrac{|\nabla u|^2}{u}\varphi dx\right)^\frac{1}{2}\left(\iB u^2 \varphi dx +  1\right)^\frac{1}{3}\notag\\
&\leq \dfrac{1}{4}\iB \dfrac{|\nabla u|^2}{u}\varphi dx + C^2\left(\iB u^2 \varphi dx + 1\right)^\frac{2}{3}\notag\\
&\leq \dfrac{1}{4}\iB \dfrac{|\nabla u|^2}{u}\varphi dx + \varepsilon\iB u^2 \varphi dx + C(\varepsilon).\label{eq3.4.7}
\end{align}
Gathering \eqref{eq:3.4.2}--\eqref{eq3.4.7} yields
\begin{align}
&\dfrac{d}{dt}\iB u\log u \varphi dx + \dfrac{3}{4}\iB \dfrac{|\nabla u|^2}{u}\varphi dx\notag\\
&\leq 8\varepsilon \iB u^2 \varphi dx + \dfrac{1}{16\varepsilon}\iB |w_t|^2 \varphi dx + \dfrac{1}{4\varepsilon}\iB v^2 \varphi dx + C(\varepsilon).
\end{align}
Similarly to the previous calculations with respect to $\iB u\log u\varphi dx$, we arrive at
\begin{align}
&\dfrac{d}{dt}\iB u\log u \varphi dx + \iB u\log u \varphi dx + \dfrac{3}{4}\iB \dfrac{|\nabla u|^2}{u}\varphi dx\notag\\
&\leq 9\varepsilon \iB u^2 \varphi dx + \dfrac{1}{16\varepsilon}\iB |w_t|^2 \varphi dx + \dfrac{1}{4\varepsilon}\iB v^2 \varphi dx + C(\varepsilon).\label{eq:3.4.8}
\end{align}
Next, multiplying the third equation of \eqref{p} by $-(\Delta w)\varphi$ and integrating by parts, we have 
\begin{align*}
\dfrac{1}{2}\dfrac{d}{dt}\iB |\nabla w|^2 \varphi dx &+ \dfrac{1}{2}\iB |\Delta w|^2 \varphi dx + 
\iB |\nabla w|^2 \varphi dx \\
&\leq \dfrac{1}{2}\iB v^2 \varphi dx - \iB w_t\nabla w \cdot \nabla \varphi dx - \iB w\nabla w \cdot \nabla \varphi dx
\end{align*}
from the H\"{o}lder inequality. Furthermore, integrating by parts after multiplying the third equation of \eqref{p} by $w_t \varphi$, we similarly get
\begin{align*}
\dfrac{1}{2}\iB |w_t|^2\varphi dx &+ \dfrac{1}{2}\dfrac{d}{dt}\iB |\nabla w |^2\varphi dx + \dfrac{1}{2}\dfrac{d}{dt}\iB w^2 \varphi dx\\
&\leq \dfrac{1}{2}\iB v^2\varphi dx - \iB w_t\nabla w \cdot \nabla \varphi dx.
\end{align*}
By multiplying each of the above inequalities by $\frac{1}{2\varepsilon}$ and then by collecting them, it holds that
\begin{align}
\dfrac{1}{2\varepsilon}\dfrac{d}{dt}\iB |\nabla w|^2 \varphi dx &+ \dfrac{1}{4\varepsilon}\dfrac{d}{dt}\iB w^2\varphi dx + \dfrac{1}{4\varepsilon}\iB |w_t|^2\varphi dx\notag\\ 
&+ \dfrac{1}{4\varepsilon}\iB |\Delta w|^2 \varphi dx + \dfrac{1}{2\varepsilon}\iB |\nabla w|^2 \varphi dx\notag\\
&\leq \dfrac{1}{2\varepsilon}\iB v^2 \varphi dx - \dfrac{1}{\varepsilon}\iB w_t \nabla w \cdot \nabla \varphi dx - \dfrac{1}{2\varepsilon}\iB w \nabla w \cdot \nabla \varphi dx\notag\\
&=\dfrac{1}{2\varepsilon}\iB v^2 \varphi dx - \dfrac{1}{2\varepsilon}\iB w \nabla w \cdot \nabla \varphi dx + III,\label{eq:3.4.9}
\end{align}
where
\begin{align*}
III &= - \dfrac{1}{\varepsilon}\iB w_t \nabla w \cdot \nabla \varphi dx.
\end{align*}
Here, integrating by parts implies 
\begin{align*}
- \dfrac{1}{2\varepsilon}\iB w \nabla w \cdot \nabla \varphi dx = \dfrac{1}{2\varepsilon}\iB w^2 \Delta \varphi dx + \dfrac{1}{2\varepsilon}\iB w \nabla w \cdot \nabla \varphi dx
\end{align*}
and thus 
\begin{align}
-\dfrac{1}{2\varepsilon}\iB w \nabla w \cdot \nabla \varphi dx &= \dfrac{1}{4\varepsilon}\iB w^2\Delta\varphi dx \leq C(\varepsilon)\label{eq:3.4.9.5}
\end{align}
due to \eqref{pro:3.2.1}. On the other hand, with respect to $III$, we see from the H\"older inequality and the Young inequality that
\begin{align}
III &\leq \dfrac{A}{\varepsilon}\iB |w_t||\nabla w|\varphi^{1-\frac{1}{n}} dx\notag\\
&\leq \dfrac{1}{16\varepsilon}\iB |w_t|^2 \varphi dx + \dfrac{4A^2}{\varepsilon}\iB |\nabla w|^2 \varphi^{1-\frac{2}{n}} dx\label{eq:3.4.10}.
\end{align}
The second term on the right-hand side of the above inequality can be rewritten as
\begin{align}
\dfrac{4A^2}{\varepsilon}\iB |\nabla w|^2 \varphi^{1-\frac{2}{n}} dx &= \dfrac{4A^2}{\varepsilon}\iB (\nabla w \cdot \nabla w) \varphi^{1-\frac{2}{n}} dx\notag\\
&=- \dfrac{4A^2}{\varepsilon}\iB w(\Delta w) \varphi^{1-\frac{2}{n}} dx - \dfrac{4A^2}{\varepsilon}\iB w \nabla w \cdot \nabla \varphi^{1-\frac{2}{n}} dx\notag\\
&=: IV + V.\label{eq:3.4.11}
\end{align}
Firstly the term $IV$ is estimated as follows:
\begin{align}
IV &\leq \dfrac{4A^2}{\varepsilon}
\iB wv \varphi^{1-\frac{2}{n}} dx - \dfrac{4A^2}{\varepsilon}\iB w w_t \varphi^{1-\frac{2}{n}}dx\notag\\
&\leq \dfrac{1}{4\varepsilon}\iB v^2 \varphi dx + \dfrac{16A^4}{\varepsilon}\iB w^2 \varphi^{1-\frac{4}{n}}dx + \dfrac{1}{•16\varepsilon}\iB|w_t|^2\varphi dx + \dfrac{64A^4}{\varepsilon}\iB w^2 \varphi^{1-\frac{4}{n}} dx\notag.
\end{align}
Since $n$ is sufficiently large, it holds from \eqref{pro:3.2.1} that
\begin{align}
IV \leq \dfrac{1}{4\varepsilon}\iB v^2 \varphi dx + \dfrac{1}{•16\varepsilon}\iB|w_t|^2\varphi dx + C(\varepsilon).\label{eq:3.4.12}
\end{align}
With respect to the term $V$, we derive from the integration by parts, Lemma \ref{lem:30}, and \eqref{pro:3.2.1} that
\begin{align}
V &= -\dfrac{1}{2}\dfrac{4A^2}{\varepsilon} \iO \nabla w^2 \cdot \nabla \varphi^{1-\frac{2}{n}} dx\notag\\
&= \dfrac{2A^2}{\varepsilon}\iO w^2 \Delta \varphi^{1-\frac{2}{n}} dx\notag\\
&\leq  \dfrac{2A^2}{\varepsilon}\|\Delta \varphi^{1-\frac{2}{n}}\|_{L^\infty(\Omega)}\iO w^2 dx\notag\\
&\leq  C(\varepsilon).\label{eq:3.4.13}
\end{align}
Here, we remark that $n$ is sufficiently large.
Combining \eqref{eq:3.4.9}--\eqref{eq:3.4.13} brings about 
\begin{align}
\dfrac{1}{2\varepsilon}\dfrac{d}{dt}\iB |\nabla w|^2 \varphi dx &+ \dfrac{1}{4\varepsilon}\dfrac{d}{dt}\iB w^2\varphi dx + \dfrac{1}{8\varepsilon}\iB |w_t|^2 \varphi dx\notag\\
&+ \dfrac{1}{2\varepsilon}\iB |\nabla w|^2 \varphi dx + \dfrac{1}{4\varepsilon}\iB |\Delta w|^2 \varphi dx\notag\\
&\leq \dfrac{3}{4\varepsilon}\iB v^2 \varphi dx + C(\varepsilon).\notag
\end{align}
By virtue of \eqref{pro:3.2.1}, we add $\frac{1}{4\varepsilon}\iO w^2 \varphi dx$ to the above inequality to obtain
\begin{align}
\dfrac{1}{2\varepsilon}\dfrac{d}{dt}\iB |\nabla w|^2 \varphi dx &+ \dfrac{1}{4\varepsilon}\dfrac{d}{dt}\iB w^2\varphi dx + \dfrac{1}{8\varepsilon}\iB |w_t|^2 \varphi dx\notag\\
&+ \dfrac{1}{2\varepsilon}\iB |\nabla w|^2 \varphi dx + \dfrac{1}{4\varepsilon}\iO w^2 \varphi dx + 
\dfrac{1}{4\varepsilon}\iB |\Delta w|^2 \varphi dx\notag\\
&\leq \dfrac{3}{4\varepsilon}\iO v^2 \varphi dx + C(\varepsilon).
\label{eq:3.4.15}
\end{align}
Finally, multiplying the second equation by $v\varphi$, and using the H\"{o}lder inequality and the Young inequality, we get 
\begin{align}
\dfrac{1}{2}\dfrac{d}{dt}\iB v^2 \varphi dx + \iB v^2 \varphi dx \leq \dfrac{1}{4\varepsilon}\iB v^2 \varphi dx + \varepsilon \iB u^2 \varphi dx.\label{eq:3.4.16}
\end{align}
Therefore introducing
\begin{align*}
E(t) := \iB u\log u \varphi dx + \dfrac{1}{2\varepsilon}\iB |\nabla w|^2 \varphi dx + \left(\dfrac{1}{4\varepsilon} + \dfrac{1}{2}\right) \iB w^2 \varphi dx
\end{align*}
and further gathering \eqref{eq:3.4.8}, \eqref{eq:3.4.15}, and \eqref{eq:3.4.16}, 
we confirm that for any $\varepsilon > 0$, there exists a positive constant $C(\varepsilon)$ such that for all $t \in (0,\infty)$,
\begin{align*}
\dfrac{d}{dt}\left(E(t) + \dfrac{1}{2}\iB v^2 \varphi dx\right) &+ E(t) + \iB v^2 \varphi dx + \dfrac{3}{4}\iB \dfrac{|\nabla u|^2}{u}\varphi dx\\
&+ \dfrac{1}{8\varepsilon}\iB |w_t|^2 \varphi dx + \dfrac{1}{4\varepsilon}\iB |\Delta w|^2 \varphi dx\\
&\leq 10\varepsilon \iB u^2\varphi dx + \dfrac{5}{4\varepsilon}\iB v^2 \varphi dx + \dfrac{1}{16\varepsilon}\iB |w_t|^2\varphi dx + C(\varepsilon).
\end{align*}
This completes the proof.
\end{proof}

In particular, we need the following inequality in order to prove Proposition \ref{prop:3.11} in next subsection. The proof is straightforward by setting $\varepsilon = \frac{5}{2}$ in Lemma \ref{lem:3.5}.
\begin{corollary}\label{rem:4}
Let $x_0 \in \overline{\Omega}, 0 < r \ll 1$, and $n$ be sufficiently large, and more let $\varphi = \varphi_{(x_0,r,n)}$ be as in Lemma \ref{lem:30}. Then it holds that for any $t \in (0,\infty)$, 
\begin{align}
\dfrac{d}{dt}\tilde{E}(t) + \tilde{E}(t) + \dfrac{3}{4}\iB \dfrac{|\nabla u|^2}{u}\varphi dx + \dfrac{1}{20}\iB |w_t|^2 \varphi dx  \leq 25 \iB u^2 \varphi dx + C,\label{eq:3.6.1}
\end{align}
where
\begin{align*}
\tilde{E}(t) := E(t) + \dfrac{1}{2}\iB v^2 \varphi dx.
\end{align*}
\end{corollary}

\subsection{Iteration argument}\label{subsec:3.3}

In the previous subsection, we aimed to derive the inequality to establish the boundedness of $L\log L$ in a localized domain. In this subsection, based on the quasi-dissipative inequality derived in the previous subsection, we utilize the Alikakos--Moser iteration scheme (see \cite{A1979}) to obtain the time-independent a priori estimates in a localized area.
Besides, it should be noted that we always assume the global-in-time solution of \eqref{p}. We obtain the global estimates in time for $\iO u^2 \varphi dx$ and $\iO v^3 \varphi dx$, which differs from \cite{NSS2000}.
\begin{proposition}\label{prop:3.11}
Let $x_0 \in \overline{\Omega}, 0 < r \ll 1$, and $n$ be sufficiently large, and more let $(u,v,w)$ be a solution of \eqref{p} in $\Omega \times [0,\infty)$. Then, if the solution $(u,v,w)$ satisfies
\begin{align}
\sup_{t > 0}\int_{B_{2r}(x_0) \cap \Omega} u\log u dx < \infty,
\end{align}
it holds that
\begin{align}
\sup_{t > 0}\|u(t)\|_{L^\infty(B_{r}(x_0) \cap \Omega)} < \infty.
\end{align}
\end{proposition}

\begin{proof}

Let $\varphi = \varphi_{(x_0,4r,n)}$ be as in Lemma \ref{lem:30}. Since the assumption signifies
\begin{align*}
\sup_{t > 0}\int_{B_{8r}(x_0) \cap \Omega} \left(u\log u + \dfrac{1}{e}\right) dx < \infty,
\end{align*}
owing to Lemma \ref{lem:3.3} (ii), we can choose sufficiently large $s > 1$ such that for all $t \in (0,\infty)$,
\begin{align*}
25 \iB u^2 \varphi \leq \dfrac{1}{4}\iB \dfrac{|\nabla u|^2}{u}\varphi dx + C_1\|u_0\|_{L^1(\Omega)}^2 + C_2,
\end{align*}
where $C_1$ is a positive constant determined by $K_{\mathrm{Sob}}$ and $A$, and $C_2$ is a positive constant determined by $K_{\mathrm{Sob}}, A$, and $\Omega$. According to Corollary \ref{rem:4} and the above inequality, we have
\begin{align*}
\dfrac{d}{dt}\tilde{E}(t) + \tilde{E}(t) + \dfrac{3}{4}\iB \dfrac{|\nabla u|^2}{u}\varphi dx + \dfrac{1}{20}\iB |w_t|^2 \varphi dx &\leq 25\iB u^2 \varphi dx + C\\
&\leq \dfrac{1}{4}\iB \dfrac{|\nabla u|^2}{u}\varphi dx + C
\end{align*}
and thus for all $t \in (0,\infty)$,
\begin{align*}
\dfrac{d}{dt}\tilde{E}(t) + \tilde{E}(t) + \dfrac{1}{2}\iB \dfrac{|\nabla u|^2}{u}\varphi dx + \dfrac{1}{20}\iB |w_t|^2 \varphi dx \leq C.
\end{align*}
This represents
\begin{align}
&\sup_{t > 0} \iB |\nabla w|^2 \varphi dx < \infty,\label{eq:3.7.1}\\
&\sup_{t > 0} \iB v^2 \varphi dx < \infty,\label{eq:3.7.2}\\
&\sup_{t > 0}\int_0^t e^{s-t} \iB |w_t|^2 \varphi dx ds < \infty.\label{eq:3.7.3}
\end{align}
Thanks to Lemma \ref{lem:30}, we can rewrite \eqref{eq:3.7.3} as
\begin{align}
\sup_{t > 0}\int_0^t e^{s-t} \int_{B_{4r}(x_0) \cap \Omega} |w_t|^2 dx ds < \infty.\label{eq:3.7.4}
\end{align}
Redefining $\varphi$ as $\varphi = \varphi_{(x_0,2r,n)}$ and using a similar approach, we also have \eqref{eq:3.7.1} and \eqref{eq:3.7.2} for such $\varphi$.

\medskip
\textbf{(Step 1)} We will first show that
\begin{align}
\sup_{t > 0}\int_{B_{2r}(x_0) \cap \Omega}|w_t|^2 dx < \infty.\label{wt}
\end{align}
Let $\varphi = \varphi_{(x_0,2r,n)}$ be as in Lemma \ref{lem:30}. Multiplying the third equation of \eqref{p} by $w_{tt}\varphi$ and integrating by parts, we have
\begin{align}
&\dfrac{1}{2}\dfrac{d}{dt}\iO |w_t|^2 \varphi dx\notag\\
 &= \iO (w_{tt}w_t)\varphi dx\notag\\
&= \iO (\Delta w_t - w_t + v_t)w_t \varphi dx\notag\\
&= -\iO |\nabla w_t|^2 \varphi dx - \iO |w_t|^2 \varphi dx - \iO w_t\nabla w_t \cdot \nabla \varphi dx + \iO v_t w_t \varphi dx.\label{wt_1}
\end{align}
Here thanks to \eqref{eq:3.7.2},
we substitute the second equation of \eqref{p} into the last term on the right-hand side of the above equality and apply the Cauchy--Schwarz and the Young inequalities to obtain
\begin{align}
\iO v_t w_t \varphi dx &\leq \iO v^2\varphi dx + \dfrac{1}{4}\iO |w_t|^2 \varphi dx + \iO u^2 \varphi dx + \dfrac{1}{4}\iO |w_t|^2 \varphi dx\notag\\
&\leq \dfrac{1}{2}\iO |w_t|^2 \varphi dx + \iO u^2 \varphi dx + C.\label{wt_2}
\end{align}
Moreover, by noting that $n$ is sufficiently large, 
integration by parts and Lemma \ref{lem:30} yield that
\begin{align}
-\iO w_t\nabla w_t \cdot \nabla \varphi dx &= -\dfrac{1}{2}\iO \nabla (w_t)^2 \cdot \nabla \varphi dx\notag\\
&= \dfrac{1}{2}\iO |w_t|^2 \Delta \varphi dx\notag\\
&\leq \dfrac{B}{2}\int_{B_{4r}(x_0) \cap \Omega} |w_t|^2 dx.\label{wt_3}
\end{align}
Combining \eqref{wt_2} and \eqref{wt_3} with \eqref{wt_1}, we achieve
\begin{align}
\dfrac{d}{dt}\iO |w_t|^2 \varphi dx + \iO |w_t|^2\varphi + 2\iO |\nabla w_t|^2 \varphi &\leq 2\iO u^2\varphi + B\int_{B_{4r}(x_0) \cap \Omega} |w_t|^2 dx + C.\label{wt_4}
\end{align}
By setting 
\begin{align*}
\mathbb{E}(t) := \tilde{E}(t) + \iO |w_t|^2 \varphi  dx,
\end{align*}
the combination of Corollary \ref{rem:4} and \eqref{wt_4} allows us to derive
\begin{align*}
\dfrac{d}{dt}\mathbb{E}(t) + \mathbb{E}(t) + \dfrac{3}{4}\iO \dfrac{|\nabla u|^2}{u}\varphi dx &\leq 27\iO u^2 \varphi dx + B\int_{B_{4r}(x_0) \cap \Omega} |w_t|^2 dx + C.
\end{align*}
In the same manner as above, we can obtain from the assumption and Lemma \ref{lem:3.3} (ii) that
\begin{align*}
\dfrac{d}{dt}\mathbb{E}(t) + \mathbb{E}(t) + \dfrac{1}{2}\iO \dfrac{|\nabla u|^2}{u}\varphi dx \leq B\int_{B_{4r}(x_0) \cap \Omega} |w_t|^2 dx + C.
\end{align*}
Therefore by multiplying $e^t$ and integrating with respect to time we infer from 
\eqref{eq:3.7.4} that for all $t \in (0,\infty)$
\begin{align*}
\mathbb{E}(t) &\leq \mathbb{E}(0)e^{-t} + B\int_0^t e^{s-t}\int_{B_{4r}(x_0) \cap \Omega} |w_t|^2 dx ds + C\\
&\leq C,
\end{align*}
which implies our claim \eqref{wt}.

From now on,  we assume $\varphi = \varphi_{(x_0,r,n)}$ throughout the proof.
The remainder of the proof falls naturally into three parts.

\medskip
\textbf{(Step 2)} We will next show that
\begin{align}
\sup_{t>0} \iB u^2 \varphi dx < \infty, \quad \sup_{t > 0} \iB v^3\varphi dx  < \infty.\label{eq:3.7.5}
\end{align}
Multiplying the first equation of \eqref{p} by $u\varphi$ and integrating by parts, we obtain
\begin{align*}
&\dfrac{1}{2}\dfrac{d}{dt}\iB u^2\varphi dx + \iB |\nabla u|^2 \varphi dx\\
&= -\iB u\nabla u \cdot \nabla \varphi dx + \iB u (\nabla w \cdot \nabla u) \varphi dx+ \iB u^2 \nabla w \cdot \nabla \varphi dx.
\end{align*}
Here, since the solution $(u,v,w)$ of \eqref{p} is positive, by means of the H\"{o}lder inequality and the third  equation of \eqref{p} we establish
\begin{align*}
\iB u (\nabla w \cdot \nabla u) \varphi dx &= \dfrac{1}{2}\iB (\nabla w \cdot \nabla u^2) \varphi dx\\
&= -\dfrac{1}{2}\iB u^2 (\Delta w) \varphi dx - \dfrac{1}{2}\iB u^2 \nabla w \cdot \nabla \varphi dx\\
&= -\dfrac{1}{2}\iB u^2 (w_t + w - v)\varphi dx - \dfrac{1}{2}\iB u^2 \nabla w \cdot \nabla \varphi dx\\
&\leq -\dfrac{1}{2} \iB u^2 w_t \varphi dx + \dfrac{1}{2}\iB u^2 v \varphi dx - \dfrac{1}{2}\iB u^2 \nabla w \cdot \nabla \varphi dx\\
&\leq \dfrac{1}{2}\left(\iB u^4 \varphi^2 dx\right)^\frac{1}{2}\left(\int_{B_{2r}(x_0) \cap \Omega} |w_t|^2 dx\right)^\frac{1}{2}\\
&\hspace{0.5cm}+ \dfrac{1}{2}\left(\iB u^3\varphi dx\right)^\frac{2}{3}\left(\iB v^3 \varphi dx\right)^\frac{1}{3} - \dfrac{1}{2}\iB u^2 \nabla w \cdot \nabla \varphi dx.
\end{align*}
Integrating by parts, we thus gain
\begin{align}
&\dfrac{1}{2}\dfrac{d}{dt}\iB u^2\varphi dx + \iB |\nabla u|^2 \varphi dx\notag\\ &\leq \dfrac{1}{2}\iB u^2
\Delta \varphi dx + \dfrac{1}{2}\iB u^2 \nabla w \cdot \nabla \varphi dx
 + \dfrac{1}{2}\left(\iB u^4 \varphi^2 dx\right)^\frac{1}{2}\left(\int_{B_{2r}(x_0) \cap \Omega} |w_t|^2 dx\right)^\frac{1}{2}\notag\\
&\hspace{0.5cm}+\dfrac{1}{2}\left(\iB u^3\varphi dx\right)^\frac{2}{3}\left(\iB v^3 \varphi dx\right)^\frac{1}{3}\notag\\
&= \dfrac{1}{2}\iB u^2 \Delta \varphi dx - \dfrac{1}{2}\iB w \nabla u^2 \cdot \nabla\varphi dx - \dfrac{1}{2}\iB u^2 w \Delta \varphi dx\notag\\
&\hspace{0.5cm} + \dfrac{1}{2}\left(\iB u^4 \varphi^2 dx\right)^\frac{1}{2}\left(\int_{B_{2r}(x_0) \cap \Omega} |w_t|^2 dx\right)^\frac{1}{2} +\dfrac{1}{2}\left(\iB u^3\varphi dx\right)^\frac{2}{3}\left(\iB v^3 \varphi dx\right)^\frac{1}{3}\notag\\
&= \dfrac{1}{2}\iB u^2(1-w) \Delta \varphi dx - \dfrac{1}{2}\iB w \nabla u^2 \cdot \nabla \varphi dx\notag\\
&\hspace{0.5cm} + \dfrac{1}{2}\left(\iB u^4 \varphi^2 dx\right)^\frac{1}{2}\left(\int_{B_{2r}(x_0) \cap \Omega} |w_t|^2 dx\right)^\frac{1}{2} +\dfrac{1}{2}\left(\iB u^3\varphi dx\right)^\frac{2}{3}\left(\iB v^3 \varphi dx\right)^\frac{1}{3}\notag\\
& =: \bf{I_1} + \bf{I_2} + \bf{I_3} + \bf{I_4},\label{eq:3.7.6}
\end{align}
where
\begin{align*}
\bf{I_1} &:= \dfrac{1}{2}\iB u^2(1-w) \Delta \varphi dx,\\
\bf{I_2} &:= - \dfrac{1}{2}\iB w \nabla u^2 \cdot \nabla \varphi dx,\\
\bf{I_3} &:= \dfrac{1}{2}\left(\iB u^4 \varphi^2 dx\right)^\frac{1}{2}\left(\int_{B_{2r}(x_0) \cap \Omega} |w_t|^2 dx\right)^\frac{1}{2},\\
\bf{I_4} &:= \dfrac{1}{2}\left(\iB u^3\varphi dx\right)^\frac{2}{3}\left(\iB v^3 \varphi dx\right)^\frac{1}{3}.
\end{align*}
Applying the H\"older inequality to $\bf{I_1}$, we deduce from \eqref{pro:3.2.1} and Lemma \ref{lem:30} that
\begin{align}
\bf{I_1} &\leq \dfrac{B}{2}\iB u^2 (w+1) \varphi^{1-\frac{2}{n}} dx\notag\\
&\leq \dfrac{1}{3}\iB u^3 \varphi dx + C.\label{eq:3.7.7}
\end{align}
Here, as noted in \eqref{eq:3.4.6} in the proof of Lemma \ref{lem:3.5}, we shall take sufficiently large $n$.
The term $\bf{I_2}$ can be handled in much the same way. Through the use of the H\"older inequality and the Young inequality we thus get, thanks to \eqref{pro:3.2.1},
\begin{align}
\bf{I_2} &\leq \iB w u |\nabla u||\nabla \varphi| dx\notag\\
&\leq A\iB w u |\nabla u|\varphi^{1-\frac{1}{n}} dx\notag\\
&\leq A\left(\iB |\nabla u|^2 \varphi dx\right)^\frac{1}{2}\left(\iB u^3 \varphi dx\right)^\frac{1}{3}\left(\iB w^6 \varphi^{1-\frac{6}{n}} dx\right)^\frac{1}{6}\notag\\
&\leq C \left(\iB |\nabla u|^2 \varphi dx\right)^\frac{1}{2}\left(\iB u^3 \varphi dx\right)^\frac{1}{3}\notag\\
&\leq \dfrac{1}{8}\iB |\nabla u|^2 \varphi dx + 2C^2\left(\iB u^3 \varphi dx\right)^\frac{2}{3}\notag\\
&\leq \dfrac{1}{8}\iB |\nabla u|^2 \varphi dx +\dfrac{1}{3}\iB u^3 \varphi dx + C.\label{eq:3.7.8}
\end{align}
We next make use of the Gagliardo--Nirenberg inequality:
\begin{align*}
\|z\|_{L^4(\Omega)} \leq K\left(\|\nabla z\|_{L^2(\Omega)}^\frac{1}{2} \|z\|_{L^2(\Omega)}^\frac{1}{2} + \|z\|_{L^2(\Omega)}\right)
\end{align*}
for $z \in H^1(\Omega)$, where $K > 0$ is a constant depending only on $\Omega$. Applying the above inequality with $z = u\varphi^\frac{1}{2}$ and employing the H\"older inequality, we derive for sufficiently large $n$ that
\begin{align}
\bf{I_3} &\leq K^2\left(\iB |\nabla (u\varphi^\frac{1}{2})|^2
dx\right)^\frac{1}{2}\left(\iB u^2 \varphi dx\right)^\frac{1}{2}\left(\int_{B_{2r}(x_0) \cap \Omega} |w_t|^2  dx\right)^\frac{1}{2}\notag\\
&\hspace{0.5cm} + K^2\left(\iB u^2 \varphi dx\right)\left(\int_{B_{2r}(x_0) \cap \Omega} |w_t|^2 dx\right)^\frac{1}{2}\notag\\
&\leq \dfrac{1}{16}\iB |\nabla (u\varphi^\frac{1}{2})|^2
dx + 4K^4\left(\iB u^2\varphi dx\right)\left(\int_{B_{2r}(x_0) \cap \Omega} |w_t|^2  dx\right)\notag\\
&\hspace{0.5cm}+\dfrac{K^2}{2}\left(\int_{B_{2r}(x_0) \cap \Omega} |w_t|^2  dx + 1 \right)\left(\iB u^2\varphi dx\right)\notag\\
&\leq \dfrac{1}{8}\iB |\nabla u|^2 \varphi dx + \dfrac{A^2}{32}\iB u^2 \varphi^{1-\frac{2}{n}} dx
+ \dfrac{K^2}{2} \iB u^2 \varphi dx\notag\\ 
&\hspace{0.5cm} + \left(4K^4 + \dfrac{K^2}{2}\right)\left(\iB u^2\varphi dx\right)\left(\int_{B_{2r}(x_0) \cap \Omega} |w_t|^2  dx\right)\notag\\
&\leq \dfrac{1}{8}\iB |\nabla u|^2 \varphi dx + \dfrac{1}{3}\iB u^3 \varphi dx + C + \dfrac{K^2}{2}\iB u^2 \varphi dx\notag\\
&\hspace{0.5cm}+\left(4K^4 + \dfrac{K^2}{2}\right)\left(\iB u^2\varphi dx\right)\left(\int_{B_{2r}(x_0) \cap \Omega} |w_t|^2  dx\right).\label{eq:3.7.9}
\end{align}
As to the term $\bf{I_4}$, we notice that the Young inequality gives
\begin{align}
\bf{I_4} &\leq \dfrac{1}{3}\iB u^3 \varphi dx + \dfrac{1}{6}\iB v^3 \varphi dx.\label{eq:3.7.10}
\end{align}
Therefore combining \eqref{eq:3.7.6}--\eqref{eq:3.7.10} yields that
\begin{align}
\dfrac{d}{dt}\iB u^2 \varphi dx + \dfrac{3}{2}\iB |\nabla u|^2 \varphi dx &\leq \dfrac{8}{3}\iB u^3 \varphi dx + \dfrac{1}{3}\iB v^3 \varphi dx + K^2\iB u^2 \varphi dx + C\notag\\
&\hspace{0.3cm}+ (8K^4 + K^2)\left(\iB u^2 \varphi dx\right)\left(\int_{B_{2r}(x_0) \cap \Omega}|w_t|^2 dx\right).
\label{eq:3.7.99}
\end{align}
Multiplying the second equation of \eqref{p} by $v^2 \varphi$ in order to obtain the global boundedness in time of $\iB v^3\varphi dx$, we have
\begin{align*}
\dfrac{1}{3}\dfrac{d}{dt}\iB v^3 \varphi dx + \iB v^3\varphi dx &= \iB uv^2 \varphi dx\\
&\leq \dfrac{2}{3}\iB v^3\varphi dx + \dfrac{1}{3}\iB u^3 \varphi dx
\end{align*}
and hence
\begin{align}
\dfrac{d}{dt}\iB v^3 \varphi dx + \iB v^3 \varphi dx \leq \iB u^3 \varphi dx.\label{eq:3.7.100}
\end{align}
Setting 
\begin{align*}
Y(t) := \iB u^2 \varphi dx + \iB v^3 \varphi dx
\end{align*}
and collecting \eqref{eq:3.7.99} and \eqref{eq:3.7.100}, we see that
\begin{align*}
\dfrac{d}{dt}Y(t) + \dfrac{3}{2}\iB |\nabla u|^2 \varphi dx + \dfrac{2}{3}\iB v^3 \varphi dx &\leq \dfrac{11}{3}\iB u^3 \varphi dx + K^2 \iB u^2\varphi dx + C\\
&\hspace{0.5cm}+C\left(\int_{B_{2r}(x_0) \cap \Omega}|w_t|^2  dx\right)\left(\iB u^2 \varphi dx\right).
\end{align*}
Thanks to \eqref{wt}, we have
\begin{align*}
\dfrac{d}{dt}Y(t) + \dfrac{3}{2}\iB |\nabla u|^2 \varphi dx + \dfrac{2}{3}\iB v^3 \varphi dx &\leq \dfrac{11}{3}\iB u^3 \varphi dx + (K^2 + C) \iB u^2\varphi dx + C.
\end{align*}
Recalling that
\begin{align*}
\sup_{t > 0} \int_{B_{2r}(x_0) \cap \Omega} \left(u\log u + \dfrac{1}{e}\right) dx < \infty, 
\end{align*}
by Lemma \ref{lem:3.3} (iii),
we can find sufficiently large $s > 1$ satisfying
\begin{align*}
\dfrac{11}{3}\iB u^3 \varphi dx \leq \dfrac{1}{2} \iB |\nabla u|^2 \varphi dx + C\|u_0\|_{L^1(\Omega)}^3,
\end{align*}
where $C$ is a positive constant depending on $K_{\mathrm{Sob}}, A$, and $\Omega$.
The Gagliardo--Nirenberg inequality and the Young inequality make it obvious that: for any $\varepsilon > 0$, there exists $C(\varepsilon) > 0$ such that
\begin{align}
\|z\|_{L^2(\Omega)}^2 \leq \varepsilon \|\nabla z\|_{L^2(\Omega)}^2 + C(\varepsilon)\|z\|_{L^1(\Omega)}^2\label{eq:gn2}
\end{align}
for all $z \in H^1(\Omega)$. By 
utilising the above inequality as $z = u\varphi^\frac{1}{2}$, we can deduce from Lemma \ref{lem:30} and \eqref{pro:3.1.1} that
\begin{align}
(K^2 + C) \iB u^2\varphi dx &\leq \varepsilon \iB |\nabla (u\varphi^\frac{1}{2})|^2 dx + C(\varepsilon)\notag\\
&\leq 2\varepsilon \iB |\nabla u|^2 \varphi dx + 2\varepsilon\iB u^2 |\nabla \varphi^\frac{1}{2}|^2 dx + C(\varepsilon)\notag\\
&\leq 2\varepsilon \iB |\nabla u|^2 \varphi dx + \dfrac{\varepsilon A^2}{2}\iB u^2 \varphi^{1-\frac{2}{n}} dx + C(\varepsilon)\notag\\
&\leq 2\varepsilon \iB |\nabla u|^2 \varphi dx + \iB u^3 \varphi dx + C\notag\\
&\leq 3\varepsilon \iB |\nabla u|^2 \varphi dx + C.\label{eq:3.7.1000}
\end{align}
Here, the last inequality in the above follows by the same method as before. By pulling together the estimates so far and choosing $\varepsilon = \frac{1}{6}$, we thus infer that
\begin{align*}
\dfrac{d}{dt}Y(t) +\dfrac{1}{2}\iB |\nabla u|^2 \varphi dx + \dfrac{2}{3}\iB v^3 \varphi dx
&\leq C.
\end{align*}
We now apply this argument \eqref{eq:3.7.1000} again, with $K^2 + C$ replaced by $1$, to obtain

\begin{align*}
&\dfrac{d}{dt}Y(t) + \iB u^2\varphi dx - C + \dfrac{2}{3}\iB v^3 \varphi dx \leq C
\end{align*}
and hence using the definition of $Y(t)$, we acquire that

\begin{align*}
\dfrac{d}{dt}Y(t) + \dfrac{2}{3}Y(t) &\leq C.
\end{align*}

Therefore by multiplying $e^{\frac{2}{3}t}$ and integrating with respect to time,  we get for all $t \in (0,\infty)$
\begin{align*}
Y(t) &\leq Y(0)e^{-\frac{2}{3}t} + C\int_0^t e^{\frac{2}{3}(s-t)} ds\\
&\leq C.
\end{align*}
This is the desired conclusion.

\medskip
\textbf{(Step 3)} We will prove 
\begin{align}
\sup_{t > 0} \|w(t)\|_{C^1(\overline{B_r(x_0) \cap \Omega})} < \infty\quad \mathrm{for\ all}\ 0 < r \ll 1.\label{eq:3.7.12.1}
\end{align}
Let $\chi_{B_{2r}(x_0)}$ be the characteristic function relative to $B_{2r}(x_0)$.
Setting $v_1 = v \chi_{B_{2r}(x_0)}, v_2 = v - v_1$, we consider the solutions $w_1, w_2$ to
\begin{equation}\notag
\begin{cases}
w_t = \Delta w -w + f \qquad &\mathrm{in}\ \Omega \times (0,\infty),\\
\dfrac{\partial w}{\partial \nu} = 0 \qquad &\mathrm{on}\ \partial \Omega \times (0,\infty),\\
w(\cdot,0)=0\qquad &\mathrm{in}\ \Omega,
\end{cases}
\end{equation}
where $f = v_1, v_2$ respectively. 
We first compute $w_2$. We are familiar with the representative formula using the fundamental solution $\Gamma$ to the heat equation (see \cite{FA1964, QS2007}) and hence we have
\begin{align*}
w_2(x,t) &= \int_0^t e^{(t-s)(\Delta_N -1)}v_2(x,s) ds\\
&= \int_0^t e^{-(t-s)} \int_{\Omega} \Gamma(x,y,t-s) v_2(y,s) dy ds\\
&= \int_0^t e^{-(t-s)}\int_{\Omega \setminus B_{2r}(x_0)} \Gamma(x,y,t-s) v(y,s) dy ds
\end{align*}
for $(x,t) \in \Omega \times (0, \infty)$. In particular, for $ x \in B_r(x_0) \cap \Omega$, we can deduce from the general property for the fundamental solution $\Gamma$ (see \cite[Section 5.3]{T1997}, \cite[Section 13 in Chapter IV]{LSU1968}) that for $ t \in (0,\infty)$,
\begin{align*}
w_2(x,t) &\leq \int_0^t e^{-(t-s)}\int_{\Omega \setminus B_{2r}(x_0)} |\Gamma(x,y,t-s)| v(y,s) dy ds\\
&\leq C \int_0^t e^{-(t-s)} \int_{\Omega \setminus B_{2r}(x_0)} (t-s)^{-1} \exp \left[-\dfrac{|x-y|^2}{C(t-s)}\right] v(y,s) dy ds\\
&= C\int_0^t e^{-(t-s)} \int_{\Omega \setminus B_{2r}(x_0)} \dfrac{|x-y|^2}{C(t-s)}\exp \left[-\dfrac{|x-y|^2}{C(t-s)}\right] \dfrac{C}{|x-y|^2} v(y,s) dy ds,
\end{align*}
where $C$ is a positive constant. Noting that $s \mapsto se^{-s}$ has the maximum for $ s \ge 0$ and 
\begin{align*}
\inf\{|x-y| ; x \in B_r(x_0), y \in \Omega \setminus B_{2r}(x_0)\} \ge r,
\end{align*}
we compute
\begin{align*}
w_2(x,t) &\leq C \int_0^t \int_{\Omega \setminus B_{2r}(x_0)} e^{-(t-s)} v(y,s) dy ds\\
&= C\int_0^t e^{-(t-s)}\int_{\Omega \setminus B_{2r}(x_0)} v(y,s) dy ds\\
&\leq C\max\{\|u_0\|_{L^1(\Omega)}, \|v_0\|_{L^1(\Omega)}\}\int_0^t e^{-(t-s)} ds\\
&\leq C\max\{\|u_0\|_{L^1(\Omega)}, \|v_0\|_{L^1(\Omega)}\}
\end{align*}
due to Proposition \ref{prop:3.1.0} (ii) and thus
\begin{align}
\sup_{t > 0} \|w_2(t)\|_{L^\infty(B_r(x_0) \cap \Omega)} < \infty.\notag
\end{align}
The same conclusion can be drawn for $D_x w_2(x,t), D_x^2 w_2(x,t)$ in $(B_r(x_0) \cap \Omega) \times (0,\infty)$. 
Therefore we infer from the Sobolev embedding  that $w_2 \in L^\infty (0,\infty ; C^1(\overline{B_r(x_0) \cap \Omega}))$.
We now turn to the case $w_1$. We consider the fractional operator $(I-\Delta_N)^\alpha$ for $\alpha \in (\frac{5}{6}, 1)$ to utilize the Sobolev embedding (see \cite[Definition 1.4.7, Theorem 1.4.8 and Theorem 1.6.1]{HD1981}). We infer from \cite[Theorem 1.3.4 and Theorem 1.4.3]{HD1981} that
\begin{align*}
\sup_{t > 0}\|w_1(t)\|_{C^1(\overline{\Omega})} &\leq \sup_{t > 0}\|(I-\Delta_N)^\alpha w_1(t)\|_{L^3(\Omega)}\\
&\leq \sup_{t > 0} \int_0^t \|(I-\Delta_N)^\alpha e^{(t-s)(\Delta_N-1)} v_1(s)\|_{L^3(\Omega)} ds\\
&\leq \sup_{t > 0} \int_0^t C (t-s)^{-\alpha} e^{-\delta (t-s)}\|v_1(s)\|_{L^3(\Omega)} ds,
\end{align*}
where $\delta \in (0,1)$.
By taking a sufficiently small $r$, we can obtain 
\begin{align*}
\sup_{t > 0}\|v_1(t)\|_{L^3(\Omega)} \leq \sup_{t>0} \iB v^3 \varphi dx < \infty
\end{align*}
owing to \eqref{eq:3.7.5}. Therefore it holds that
\begin{align*}
\sup_{t > 0}\|w_1(t)\|_{C^1(\overline{\Omega})} < \infty.
\end{align*}
Hence, the linearity of the homogeneous equation yields our claim \eqref{eq:3.7.12.1}.

\medskip
\textbf{(Step 4)} We finally show the conclusion using the Alikakos--Moser iteration scheme (see \cite{A1979}).
Multiplying $u^p \varphi^{p+1}$ ($p \ge 1$) by the first equation of \eqref{p}, we get
\begin{align}
\dfrac{d}{dt}\dfrac{1}{p+1}\iB |u\varphi|^{p+1} dx &= -\iB \nabla u \cdot \nabla (u^p \varphi^{p+1}) dx + \iB u \nabla w \cdot \nabla (u^p\varphi^{p+1}) dx\notag\\
&=: -\textbf{I} + \textbf{II}.\label{eq:3.7.12}
\end{align}
Firstly the term \textbf{I} can be computed as follows:
\begin{align*}
\textbf{I} &= \iB (p u^{p-1}\varphi^{p+1}\nabla u + u^p\nabla (\varphi^{p+1})) \cdot \nabla u dx\\
&= \dfrac{4p}{(p+1)^2}\iB |\nabla (u^{\frac{p+1}{2}})|^2 \varphi^{p+1} dx + \dfrac{1}{p+1}\iB \nabla (u^{p+1}) \cdot \nabla (\varphi^{p+1}) dx\\
&= \dfrac{4p}{(p+1)^2}\iB |\nabla (u^{\frac{p+1}{2}})|^2 \varphi^{p+1} dx + \dfrac{4}{p+1} \iB \varphi^{\frac{p+1}{2}} \nabla (u^{\frac{p+1}{2}}) \cdot \nabla (\varphi^{\frac{p+1}{2}}) u^\frac{p+1}{2} dx\\
&= \left\{\dfrac{4p}{(p+1)^2} - \dfrac{2}{p+1}\right\} \iB |\nabla (u^\frac{p+1}{2})|^2 \varphi^{p+1} dx + \dfrac{2}{p+1}\iB |\nabla (u^\frac{p+1}{2}\varphi^\frac{p+1}{2})|^2 dx\\
&\hspace{0.5cm}- \dfrac{2}{p+1}\iB u^{p+1} | \nabla (\varphi^\frac{p+1}{2})|^2 dx.
\end{align*}
Observing that
\begin{align*}
\dfrac{4p}{(p+1)^2} - \dfrac{2}{p+1} = \dfrac{2(p-1)}{(p+1)^2} \ge 0
\end{align*}
due to $p \ge 1$, we can infer from the H\"{o}lder inequality and Lemma \ref{lem:30} that
\begin{align}
\textbf{I} &\ge \dfrac{2}{p+1}\iB |\nabla (u^\frac{p+1}{2}\varphi^\frac{p+1}{2})|^2 dx - 
\dfrac{2}{p+1}\iB u^{p+1} | \nabla (\varphi^\frac{p+1}{2})|^2 dx\notag\\
&\ge \dfrac{2}{p+1}\iB |\nabla (u^\frac{p+1}{2}\varphi^\frac{p+1}{2})|^2 dx - 
\dfrac{A^2(p+1)}{2}\iB u^{p+1} \varphi^{p-1} \varphi^{2-\frac{2}{n}} dx\notag\\
&= \dfrac{2}{p+1}\iB |\nabla (u^\frac{p+1}{2}\varphi^\frac{p+1}{2})|^2 dx - 
\dfrac{A^2(p+1)}{2}\iB |u\varphi|^{p+1-\frac{2}{n}} u^\frac{2}{n} dx\notag\\
&\ge \dfrac{2}{p+1}\iB |\nabla (u^\frac{p+1}{2}\varphi^\frac{p+1}{2})|^2 dx -
\dfrac{A^2(p+1)}{2} \left( \iB |u\varphi|^{1 + \frac{n}{n-2}p} dx\right)^\frac{n-2}{n}\|u_0\|_{L^1(\Omega)}^\frac{2}{n}.\label{eq:3.7.13}
\end{align}
Secondly as to the term \textbf{II}, \eqref{eq:3.7.12.1} implies that
\begin{align}\textbf{II} &= \iB u \nabla w \cdot \nabla (u^p\varphi^{p+1}) dx\notag\\
&\leq C \iB |u \nabla (u^p\varphi^{p+1})| dx\notag\\
&\leq C\left\{ \dfrac{p}{p+1}\iB \left| \nabla (u\varphi)^{p+1}\right| dx + \iB u^{p+1} \varphi^p |\nabla \varphi| dx\right\}\notag\\
&= C\dfrac{2p}{p+1}\iB (u\varphi)^\frac{p+1}{2}|\nabla (u\varphi)^\frac{p+1}{2}| dx
 + CA \iB (u\varphi)^{p+1 - \frac{1}{n}} u^\frac{1}{n}dx\notag\\
&\leq 2C \iB (u\varphi)^\frac{p+1}{2}|\nabla (u\varphi)^\frac{p+1}{2}| dx
 + CA \iB (u\varphi)^{p+1 - \frac{1}{n}} u^\frac{1}{n} dx\notag\\
&\leq \dfrac{1}{p+1}\iB |\nabla (u\varphi)^\frac{p+1}{2}|^2 dx + C^2(p+1) \iB (u\varphi)^{p+1} dx\notag\\
&\hspace{0.5cm}+ CA\|u_0\|_{L^1(\Omega)}^\frac{1}{n}\left(\iB |u\varphi|^{1 + \frac{n}{n-1}p} dx\right)^\frac{n-1}{n}.\label{eq:3.7.14}
\end{align}
For simplicity of notation, we will write $\uv$ instead of $u\varphi$. Consequently, gathering \eqref{eq:3.7.12}--\eqref{eq:3.7.14} results in 
\begin{align}
&\dfrac{d}{dt}\iB |\uv|^{p+1} dx + \iB |\nabla (\uv)^\frac{p+1}{2}|^2 dx\notag\\
&\leq C(p+1)^2 \iB |\uv|^{p+1} dx\notag\\
&\hspace{0.5cm} + C(p+1)^2 \left\{\left(\iB |\uv|^{1 + \frac{n}{n-2}p} dx\right)^\frac{n-2}{n} + \left(\iB |\uv|^{1+\frac{n}{n-1}p}dx\right)^\frac{n-1}{n}\right\},\label{eq:3.7.15}
\end{align}
where $C > 0$ is a constant independent of $p$.   
Our goal is to pass to the limit and obtain the desired $L^\infty$-bound by controlling carefully the constants on the right-hand side of the above inequality. For this purpose, we shall make much use of the following Gagliardo--Nirenberg inequality: let $q \in [1, \infty)$. Then
\begin{align}
\|z\|_{L^q(\Omega)} \leq K (\|\nabla z\|_{L^2(\Omega)}^2 + \|z\|_{L^2(\Omega)}^2)^{\frac{1-\frac{1}{q}}{2}}\|z\|_{L^1(\Omega)}^\frac{1}{q}\label{eq:3.7.15.1}
\end{align}
for $z \in H^1(\Omega)$, where $K$ is a positive constant independent of $q$. The independence of $K$ from $q$
plays an important role in the subsequent analysis.
In order to use the above inequality for the right-hand side of \eqref{eq:3.7.15}, we need to appropriately determine $q$. Hence, we define $q$ to satisfy 
\begin{align*}
\dfrac{p+1}{2}q = 1 + \dfrac{n}{n-2}p,
\end{align*}
that is 
\begin{align*}
q = \dfrac{2(n-2) + 2np}{(p+1)(n-2)}.
\end{align*}
If $n$ is sufficiently large, it holds $q \in [2,3]$ as is easy to check. Applying the Gagliardo--Nirenberg inequality \eqref{eq:3.7.15.1} with $z = \uv^\frac{p+1}{2}$ and $q$ defined above, we thus have
\begin{align*}
&C(p+1)^2 \left(\iB |\uv|^{1 + \frac{n}{n-2}p} dx\right)^\frac{n-2}{n}\\
&=C(p+1)^2 \left(\iB |\uv^\frac{p+1}{2}|^q dx\right)^\frac{\beta}{q}\\
&\leq C(p+1)^2 K^\beta \left(\iB |\nabla (\uv)^\frac{p+1}{2}|^2 dx + \iB |\uv|^{p+1} dx \right)^{\frac{1-\frac{1}{q}}{2}\beta} \left(\iB |\uv|^\frac{p+1}{2} dx\right)^\frac{\beta}{q}\\
&\leq C(K+1)^2(p+1)^2 \left(\iB |\nabla (\uv)^\frac{p+1}{2}|^2 dx + \iB |\uv|^{p+1} dx + 1\right)^{\frac{q-1}{q}} \left(\iB |\uv|^\frac{p+1}{2} dx + 1\right)^\frac{2}{q},
\end{align*}
where $\beta = q \frac{n-2}{n} = \frac{2(n-2) + 2np}{n(p+1)} \leq 2$.
Since $ q \in [2,3]$, we can deduce from the Young inequality that
\begin{align}
&C(p+1)^2 \left(\iB |\uv|^{1 + \frac{n}{n-2}p} dx\right)^\frac{n-2}{n}\notag\\
&\leq \dfrac{1}{6}\times\dfrac{q-1}{q}\left(\iB |\nabla(\uv)^\frac{p+1}{2}|^2 dx + \iB |\uv|^{p+1} dx + 1\right)\notag\\
&\hspace{0.5cm}+ \left[C(K+1)^2(p+1)^2\right]^q \times 6^{q-1} \dfrac{1}{q}\left(\iB |\uv|^\frac{p+1}{2} dx + 1\right)^2\notag\\
&\leq \dfrac{1}{6}\left(\iB |\nabla(\uv)^\frac{p+1}{2}|^2 dx + \iB |\uv|^{p+1} dx + 1\right)
+ C(p+1)^6 \left(\iB |\uv|^\frac{p+1}{2} dx + 1\right)^2.\label{eq:3.7.16}
\end{align}
Here, we remark that the positive constant $C$ in the above inequality is independent of $p,q$ thanks to the inequality: $\frac{6^{q-1}}{q} \leq \frac{6^{3-1}}{2}$.
Applying the same manner for $\tilde{q}$ satisfying
\begin{align*}
\dfrac{p+1}{2}\tilde{q} = 1 + \dfrac{n}{n-1}p,
\end{align*} 
we have
\begin{align}
&C(p+1)^2 \left(|\uv|^{1 + \frac{n}{n-1}p} dx \right)^\frac{n-1}{n}\notag\\
&= C(p+1)^2 \left(\iB |\uv^\frac{p+1}{2}|^{\tilde{q}} dx\right)^\frac{\gamma}{\tilde{q}}\notag\\
&\leq C(K+1)^2(p+1)^2\left(\iB |\nabla (\uv)^\frac{p+1}{2}|^2 dx + \iB |\uv|^{p+1} dx + 1\right)^\frac{\tilde{q}-1}{\tilde{q}} \left(\iB |\uv|^\frac{p+1}{2} dx + 1\right)^\frac{2}{\tilde{q}}\notag\\
&\leq \dfrac{1}{6}\left(\iB |\nabla (\uv)^\frac{p+1}{2}|^2 dx + \iB |\uv|^{p+1} dx + 1\right) 
+ C(p+1)^6 \left(\iB |\uv|^\frac{p+1}{2} dx + 1\right)^2,
\end{align}
where $\tilde{q} \in [2,3]$ and $\gamma = \frac{2(n-1) + 2np}{n(p+1)} \leq 2$.\label{eq:3.7.17}
In addition, using the Gagliardo--Nirenberg inequality \eqref{eq:3.7.15.1} as $z = \uv^\frac{p+1}{2}$ and $q = 2$, we get
from the similar argument above that
\begin{align}
&C(p+1)^2 \iB |\uv|^{p+1} dx\notag\\
&= C(p+1)^2 \left(\iB |\uv^\frac{p+1}{2}|^2 dx\right)^{\frac{1}{2} \times 2}\notag\\
&\leq \dfrac{1}{6}\left(\iB |\nabla(\uv)^\frac{p+1}{2}|^2 dx + \iB |\uv|^{p+1} dx\right)
 + C(K+1)^2(p+1)^4\left(\iB |\uv|^\frac{p+1}{2} dx\right).\label{eq:3.7.18}
\end{align}
Combining \eqref{eq:3.7.15} and \eqref{eq:3.7.16}--\eqref{eq:3.7.18} yields
\begin{align*}
&\dfrac{d}{dt}\iB |\uv|^{p+1} dx + \dfrac{1}{2}\iB |\nabla (\uv)^\frac{p+1}{2}|^2 dx\\
&\leq \dfrac{1}{2}\iB |\uv|^{p+1} dx + C(p+1)^6 \left(\iB |\uv|^\frac{p+1}{2} dx + 1\right)^2.
\end{align*}
Applying \eqref{eq:gn2} with $\varepsilon = \frac{1}{2}$ and $z = \uv^\frac{p+1}{2}$, we can assert that
\begin{align*}
\dfrac{d}{dt} \iB |\uv|^{p+1} dx + \dfrac{1}{2}\iB |\uv|^{p+1} dx \leq C(p+1)^6 \left(\iB |\uv|^\frac{p+1}{2} dx + 1\right)^2,
\end{align*}
where $C$ is a positive constant independent of $p$. We are now in a position to show the conclusion. The above inequality gives that for $t \in (0,T)$ with some $T > 0$,
\begin{align*}
\iB |\uv|^{p+1} dx + 1 \leq \left(\iB u_0^{p+1} \varphi^{p+1} dx + 1\right) + \sup_{t \in (0,T)}\left(\iB |\uv|^\frac{p+1}{2} dx + 1\right)^2 \times C(p+1)^6.
\end{align*}
Setting $2^{k+1} := p+1$ and
\begin{align*}
\Phi_k := \sup_{t \in (0,T)} \left(\iB |\uv|^{2^k} dx +1 \right),
\end{align*}
we evaluate as follows:
\begin{align*}
\Phi_{k+1} &\leq \|u_0\|_{L^\infty(\Omega)}^{p+1}|\Omega| + 1 + C(p+1)^6\Phi_k^2\\
&\leq C\max\{(\|u_0\|_{L^\infty(\Omega)} + 1)^{2^{k+1}}, 2^{6(k+1)}\Phi_k^2\}\\
&\leq C 2^{6(k+1)} \max\{\Phi_k^2, (\|u_0\|_{L^\infty(\Omega)} + 1)^{2^{k+1}}\}
\end{align*}
for $k \ge 1$. By induction on $k$, we obtain 
\begin{align*}
\Phi_{k+1} \leq C^{\sum_{l=1}^k 2^{l-1}} 2^{\sum_{l=1}^k 6(l+1)2^{k-l}}\max\{\Phi_1^{2^k}, (\|u_0\|_{L^\infty(\Omega)} + 1)^{2^{k+1}})\}
\end{align*}
and thus
\begin{align*}
\sup_{t \in (0,T)}\left(\iB |\uv|^{2^{k+1}} dx\right)^\frac{1}{2^{k+1}} 
&\leq \Phi_{k+1}^\frac{1}{2^{k+1}}\\
&\leq C^\frac{2^k -1}{2^{k+1}} 2^{6\sum_{l=1}^k \frac{(l+1)2^{k-l}}{2^{k+1}}}\max\{\Phi_1^\frac{2^k}{2^{k+1}}, (\|u_0\|_{L^\infty(\Omega)} + 1)\}\\
&\leq C^\frac{2^k -1}{2^{k+1}} 2^{6\sum_{j=0}^\infty j 2^{-j}} \max\{\Phi_1^\frac{1}{2}, (\|u_0\|_{L^\infty(\Omega)} + 1)\}.
\end{align*}
Letting $k \to \infty$, we can conclude from \eqref{eq:3.7.5} that
\begin{align*}
\sup_{t \in (0,T)}\|\uv(t)\|_{L^\infty(\Omega)} &\leq C\max\{\Phi_1^\frac{1}{2}, (\|u_0\|_{L^\infty(\Omega)} + 1)\}\\
&\leq C \max\left\{\left(\sup_{t > 0}\iB u^2 \varphi dx + 1\right)^\frac{1}{2}, (\|u_0\|_{L^\infty(\Omega)} + 1)\right\}\\
&\leq C,
\end{align*} 
where $C$ is independent of $T$.
By the definition of $u_\varphi$ and Lemma \ref{lem:30}, it holds that
\begin{align*}
\sup_{t \in (0,T)} \|u(t)\|_{L^\infty(B_r(x_0) \cap \Omega)} \leq C,
\end{align*}
which completes the proof by pushing $T \to \infty$.
\end{proof}

\subsection{Proof of Theorem \ref{th:2}}\label{subsec:3.4}
The following proposition ensures that the localized mass is always lifted from below near grow-up points, which gives Theorem \ref{th:2}. The basic idea of the proof is an $\varepsilon$-regularity theorem, meaning that if the mass is less than sufficiently small locally in space, solutions achieve uniform boundedness in time locally in space.
\begin{proposition}\label{prop:5.1}
Let $x_0 \in \mathcal{G}$ and $0 < r \ll 1$. Then
\begin{align}
\limsup_{t \to \infty} \int_{B_r(x_0) \cap \Omega} u(x,t) dx \ge \dfrac{1}{200K^2_{\mathrm{Sob}}},\label{eq:4.1}
\end{align}
where $K_\mathrm{Sob}$ is the positive constant in the inequality \eqref{sobine:1}.
\end{proposition}

\begin{proof}
Give $\varphi = \varphi_{(x_0,r,n)}$ introduced in Lemma \ref{lem:30} and suppose the assertion of this proposition is false. Namely, we assume that it holds that for sufficiently large time $t$,
\begin{align*}
\int_{B_r(x_0) \cap \Omega} u(x,t) dx < \dfrac{1}{200K_{\mathrm{Sob}}^2}.
\end{align*} 
Combining the above estimate with Lemma \ref{lem:3.3} (i), we can deduce that for sufficiently large time $t$,
\begin{align*}
25\iB u^2 \varphi dx &< \dfrac{50K_{\mathrm{Sob}}^2}{200K_{\mathrm{Sob}}^2}\iB \dfrac{|\nabla u|^2}{u}\varphi dx + 25 K_{\mathrm{Sob}}^2\left(\dfrac{A^2}{2} + 1\right)\|u_0\|_{L^1(\Omega)}^2\\
&= \dfrac{1}{4}\iB \dfrac{|\nabla u|^2}{u}\varphi dx + 25 K_{\mathrm{Sob}}^2\left(\dfrac{A^2}{2} + 1\right)\|u_0\|_{L^1(\Omega)}^2.
\end{align*}
Therefore, Corollary \ref{rem:4} leads to 
\begin{align*}
\dfrac{d}{dt}\tilde{E}(t) + \tilde{E}(t) \leq C
\end{align*}
for sufficiently large $t$ and thus
\begin{align*}
\sup_{t > 0} \int_{B_r(x_0) \cap \Omega} u\log u dx < \infty.
\end{align*}
Owing to Proposition \ref{prop:3.11}, we have
\begin{align*}
\sup_{t > 0} \|u(t)\|_{L^\infty(B_{\frac{r}{2}}(x_0) \cap \Omega)} < \infty.
\end{align*}
This contradicts our assumption that $x_0$ is a grow-up point. 
\end{proof}

\begin{proof}[Proof of Theorem \ref{th:2}]
Since the initial data $(u_0,v_0,w_0)$ is a triplet of radially symmetric functions, the corresponding solution $(u,v,w)$ of \eqref{p} is
also radially symmetric. Suppose there exists a grow-up point $x_0 \not= 0$. We give $m \in \N$ satisfying
\begin{align*}
200K^2_{\mathrm{Sob}}\|u_0\|_{L^1(\Omega)} + 1 \leq m
\end{align*}
and furthermore take $0 < r \ll 1$ and $\{x_i\}_{i = 1}^{i = m} \subset \overline{\Omega}$ such that
\begin{align*}
|x_0| = |x_i|\ (1 \leq i \leq m),\quad B_r(x_i) \cap B_r(x_j) = \emptyset\ \mathrm{for}\ i\not= j.
\end{align*}
Now we have
\begin{align*}
\int_{B_r(x_i) \cap \Omega} u(x,t) dx = \int_{B_r(x_0) \cap \Omega} u(x,t) dx
\end{align*}
for $t \in (0,\infty)$ and $1 \leq i \leq m$ due to the radial symmetry. Here, it follows from \eqref{eq:4.1} that there exists
$\{t_k\} \subset (0,\infty)$ such that $t_k \to \infty$ as $k \to \infty$ and 
\begin{align*}
\int_{B_r(x_0) \cap \Omega} u(x,t_k) dx \ge \dfrac{1}{200K_{\mathrm{Sob}}^2}.
\end{align*}
Consequently, it holds from \eqref{pro:3.1.1}
that
\begin{align*}
\dfrac{m}{200K_{\mathrm{Sob}}^2} < \sum_{i =1}^m \int_{B_r(x_i) \cap \Omega} u(x,t_k) dx \leq \iB u(x,t_k) dx = \|u_0\|_{L^1(\Omega)} \leq \dfrac{m-1}{200K_{\mathrm{Sob}}^2},
\end{align*}
which is not reasonable. We thus conclude that the grow-up point is the only origin.
\end{proof}

\subsection{Proof of Theorem \ref{th:1}}\label{subsec:3.5}
In this subsection, we study the more detailed behavior of the radially symmetric grow-up solution $(u,v,w)$ to \eqref{p} near the grow-up point $x_0 = 0$. We first give a higher regularity of $(u,v,w)$ in the area excluding the origin. We present the following lemma by employing the properties of the Neumann heat semigroup. To shorten, we will write $B_r$ instead of $B_r(0)$ for $ 0 < r \ll 1$.

\begin{lemma}\label{lem:6.1}
Let $(u,v,w)$ be a radially symmetric grow-up solution of \eqref{p}. Then for any $0 < r \ll 1$ and $\alpha \in (0,1)$, there exists a positive constant $C(r,\alpha)$ such that
\begin{align}
\|(u,w)\|_{C^{2+\alpha, 1 + \frac{\alpha}{2}}((\overline{\Omega} \setminus B_r) \times (0,\infty))} \leq C(r,\alpha),\quad \|v\|_{C^{\alpha, 1+\frac{\alpha}{2}}((\overline{\Omega} \setminus B_r) \times (0,\infty))} \leq C(r,\alpha).\label{eq:6.0.1}
\end{align}
\end{lemma}

\begin{proof}
It holds from the definition of the grow-up point and Theorem \ref{th:2} that for any $0 < r \ll 1$, there exists a  positive constant $C(r)$ depending on $r$ such that
\begin{align}
\sup_{t > 0}\|u(t)\|_{L^\infty(\Omega \setminus B_r)} \leq C(r).\label{eq:6.1}
\end{align}
Moreover, solving the second equation of \eqref{p} for $v$ yields
\begin{align}
\sup_{t > 0}\|v(t)\|_{L^\infty(\Omega \setminus B_r)} \leq C(r)\label{eq:6.1.5}
\end{align}
for $0 < r \ll 1$. We can now proceed analogously to Step 2 in the proof of Proposition \ref{prop:3.11}, which implies that
\begin{align}
\sup_{t > 0} \|w(t)\|_{C^1(\overline{\Omega} \setminus B_{2r})} \leq C(r).\label{eq:6.2}
\end{align}
Let $\varphi = \varphi_{(0,2r,n)}$ be as in Lemma \ref{lem:30} and $\psi$ be defined by $\psi = 1 - \varphi$. Then $\psi \in C^\infty(\R^2)$ satisfies 
\begin{align*}
  &\psi(x)=
  \begin{cases}
    1\quad \mathrm{in}\ \Omega \setminus B_{4r}, \\
    0\quad \mathrm{in}\ B_{2r},
  \end{cases}\\
  &0 \leq \psi \leq 1\ \mathrm{in}\ \R^2,\\
  &\dfrac{\partial \psi}{\partial \nu} = 0\ \mathrm{on}\ \partial \Omega.
\end{align*}
Multiplying the first equation of \eqref{p} by $\psi$, we have
\begin{align*}
(u\psi)_t &= (\Delta u )\psi - (\nabla \cdot f)\psi\\
&=\Delta (u\psi) - 2\nabla u \cdot \nabla \psi - u\Delta \psi - \nabla \cdot (f \psi) + f \cdot \nabla \psi\\
&= \Delta (u\psi) - \nabla \cdot (f\psi + 2u\nabla \psi) + u\Delta \psi + f \cdot \nabla \psi
\end{align*}
with the notation $f := u \nabla w \in L^\infty(0,\infty ; L^\infty(\Omega \setminus B_{2r})^2)$. Owing to \eqref{eq:6.1} and \eqref{eq:6.2}, we obtain
\begin{align*}
f\psi + 2u \nabla \psi \in L^\infty(0,\infty ; L^\infty(\Omega)^2),\quad u\Delta \psi + f\cdot \nabla \psi \in L^\infty(0,\infty ; L^\infty(\Omega)).
\end{align*}
We observe that the map $s \to \|\nabla e^{(t-s)A}\|$ belongs to $L^1(0,t)$ for any $t \in (0,\infty)$, where $A$ denotes the realization of the Laplacian in $L^\infty(\Omega)$ with the $0$-Neumann boundary condition
(see \cite[Chapter 3]{LA2015}). This allows us to apply \cite[Theorem 5.1.17]{LA2015} to the above equation. Therefore, it holds that for $\varepsilon > 0$ and
$\alpha \in (0,1)$, there exists $C(\varepsilon, \alpha) > 0$ such that
\begin{align*}
\|u\|_{C^{\alpha, \frac{\alpha}{2}}((\overline{\Omega} \setminus B_{4r}) \times [\varepsilon,\infty))} &\leq 
C(\varepsilon, \alpha)(\varepsilon^{-\frac{\alpha}{2}}\|u_0\psi\|_{L^\infty(\Omega)} + \|f\psi + 2u \nabla \psi\|_{L^\infty(\Omega)}\\
&\hspace{0.5cm} + \|u\Delta \psi + f\cdot \nabla \psi\|_{L^\infty(\Omega)}).
\end{align*}
This implies that there exists $C(r,\alpha) > 0$ such that
\begin{align*}
\|u\|_{C^{\alpha, \frac{\alpha}{2}}((\overline{\Omega} \setminus B_{4r}) \times (0,\infty))} \leq C(r,\alpha)
\end{align*}
and additionally due to the second equation 
\begin{align}
\|v\|_{C^{\alpha,\frac{\alpha}{2}}((\overline{\Omega} \setminus B_{4r}) \times (0,\infty))} \leq C(r, \alpha)\label{eq:6.1.6}
\end{align}
for any $0 < r \ll 1$.
Likewise, multiplying the third equation by a cut-off function $\psi$ reconstructing to satisfy $\psi = 1$ in $\Omega \setminus B_ {8r}$ and $\psi = 0$ in $B_{4r}$, we infer that
\begin{align}
(w\psi)_t &= \Delta(w\psi) - w\psi -2 \nabla w \cdot \nabla \psi - w\Delta\psi + v\psi\notag\\
&= \Delta (w\psi) - w\psi -2\nabla \cdot (w\nabla \psi) + w\Delta \psi + v\psi.\label{eq:6.1.7}
\end{align}
Now, since \eqref{eq:6.1.5} and \eqref{eq:6.2} yield that
\begin{align*}
w \nabla \psi \in L^\infty(0,\infty; L^\infty(\Omega)^2),\quad w \Delta \psi + v\psi \in L^\infty(0,\infty; L^\infty(\Omega)),
\end{align*}
we apply \cite[Theorem 5.1.17]{LA2015} for the above equation \eqref{eq:6.1.7} again to get
\begin{align*}
\|w\|_{C^{\alpha, \frac{\alpha}{2}}((\overline{\Omega} \setminus B_{8r}) \times (0,\infty))} \leq C(r, \alpha).
\end{align*}
Thanks to the above regularity and \eqref{eq:6.1.6}, by employing \cite[Theorem 5.1.20]{LA2015} to the equation \eqref{eq:6.1.7},
we can deduce that for any 
$\alpha \in (0,1)$,
\begin{align*}
\|w\|_{C^{2+\alpha, 1+\frac{\alpha}{2}}((\overline{\Omega} \setminus B_{16r}) \times (0,\infty))} \leq C(r, \alpha).
\end{align*}
This follows by the above method as in $u$. In addition, since $(u,v,w)$ is a classical solution of \eqref{p}, we infer from the second equation of \eqref{p} that
\begin{align*}
\|v_t\|_{C^{\alpha,\frac{\alpha}{2}}((\overline{\Omega} \setminus B_{4r}) \times (0,\infty))} &=
\|-v + u\|_{C^{\alpha,\frac{\alpha}{2}}((\overline{\Omega} \setminus B_{4r}) \times (0,\infty))}\\
&\leq \|v\|_{C^{\alpha,\frac{\alpha}{2}}((\overline{\Omega} \setminus B_{4r}) \times (0,\infty))} + \|u\|_{C^{\alpha,\frac{\alpha}{2}}((\overline{\Omega} \setminus B_{4r}) \times (0,\infty))}\\
&\leq C(r,\alpha).
\end{align*}
This signifies 
\begin{align*}
\|v\|_{C^{\alpha,1+\frac{\alpha}{2}}((\overline{\Omega} \setminus B_{4r}) \times (0,\infty))} \leq C(r,\alpha).
\end{align*}
Consequently, owing to the arbitrariness of $r$, our claim is proved.
\end{proof}
We already introduced the localized Lyapunov functional $\F_\varphi$ and obtained the equality \eqref{eq:4.0.1}.
In \cite{NSS2000}, since $T_\mathrm{max}$ is finite, the combination of the relation \eqref{eq:4.0.1} and Lemma \ref{lem:6.1} immediately brings about the boundedness in $(0,T_\mathrm{max})$ of the localized Lyapunov functional. However, the same argument as \cite{NSS2000} does not imply that
the localized Lyapunov functional is uniformly bounded in time. 
The key to show Theorem \ref{th:1} is the following lemma,
which means the uniform boundedness in time of the localized Lyapunov functional.
The main idea of the proof, which is not present in \cite{NSS2000}, is to get the estimate as follows:
\begin{align*}
\F_\varphi(t) \leq \F(t) + C,
\end{align*}
where $C$ is a positive constant and $\F$ is the Lyapunov functional in Lemma \ref{lem:3.41}.

\begin{lemma}\label{lem:6.2}
Let $ 0 < r \ll 1$, and $n$ be sufficiently large, and more let $\varphi = \varphi_{(0,r,n)}$ be as in Lemma \ref{lem:30}. Then there exists a positive constant $C$ such that for $ t \in (0,\infty)$,
\begin{align*}
\dfrac{d}{dt}\F_\varphi(t) + \F_\varphi(t) + \D_\varphi (t) \leq C,
\end{align*}
where $\F_\varphi, \D_\varphi$ are introduced in Lemma \ref{lem:3.4}.
In addition, it holds that
\begin{align}
\sup_{t > 0} \F_\varphi (t) < \infty.\label{eq:6.2.1}
\end{align}
\end{lemma}

\begin{proof}
We recall \eqref{eq:4.0.1}:
\begin{align*}
\dfrac{d}{dt}\F_\varphi(t) + \D_\varphi(t) = \dfrac{d}{dt}\iB u\varphi dx + \Re(u,w,\varphi)
\end{align*}
for $t \in (0,\infty)$. Noticing that it follows from the first equation of \eqref{p} that
\begin{align*}
\dfrac{d}{dt}\iB u\varphi dx = - \iB (\nabla u - u \nabla w)\cdot\nabla \varphi dx,
\end{align*}
we can assert that there exists a positive constant $C$ such that
\begin{align*}
\sup_{t > 0}\left|\dfrac{d}{dt}\iB u\varphi dx\right| \leq C, \quad \sup_{t>0}|\Re(u,w,\varphi)| \leq C
\end{align*}
owing to Lemma \ref{lem:30} and \eqref{eq:6.0.1}. We thus have for $t \in (0,\infty)$,
\begin{align*}
\dfrac{d}{dt}\F_\varphi(t) + \D_\varphi(t) \leq C.
\end{align*}
To obtain uniform estimates in time, we add $\F_\varphi(t)$ to the above inequality to evaluate
\begin{align*}
&\dfrac{d}{dt}\F_\varphi(t) + \F_\varphi(t) + \D_\varphi(t)\\
&\leq C + \F_\varphi(t)\\
&= C + \iB (u\log u)\varphi dx + \dfrac{1}{2}\iB (|w_t|^2 + |\nabla w|^2 + w^2)\varphi dx - \iB uw \varphi dx\\
&\leq C + \iB \left(u\log u + \dfrac{1}{e}\right)\varphi dx + \dfrac{1}{2}\iB (|w_t|^2 + |\nabla w|^2 + w^2)\varphi dx - \iB uw \varphi dx\\
&\leq C + \dfrac{1}{e}|\Omega| + \iB u\log u dx + \dfrac{1}{2}\iB (|w_t|^2 + |\nabla w|^2 + w^2) dx - \iB uw \varphi dx.
\end{align*}
Now we focus on the term $- \iB uw \varphi dx$. Invoking \eqref{eq:6.0.1} yields 
\begin{align*}
- \iB uw \varphi dx &= - \int_{B_r}uw\varphi dx - \int_{\Omega \setminus B_r} uw \varphi dx\\
&\leq - \int_{B_r} uw dx + C\\
&= - \iB uw dx + \int_{\Omega \setminus B_r} uw dx + C\\
&\leq - \iB uw dx + C.
\end{align*}
Combining these inequalities and Lemma \ref{lem:3.41}, we conclude that for $ t \in (0,\infty)$,
\begin{align*}
\dfrac{d}{dt}\F_\varphi(t) + \F_\varphi(t) + \D_\varphi(t) &\leq C + \F(t)\\
& \leq C + \F(0).
\end{align*}
The rest of the proof is straightforward so that the lemma follows.
\end{proof}

Here, the following concentration lemma is an inequality modified by \cite{NSS2000} in order to study chemotactic collapse.

\begin{lemma}[{\cite[Lemma 5.3]{NSS2000}}]\label{lem:6.3}
Let $a > 0$, $0 < r \ll 1$, and $\varphi = \varphi_{(0,r,n)}$ be as in Lemma \ref{lem:30}. Then it holds that there exists a positive constant $C$ such that
for $t \in (0,\infty)$,
\begin{align}
\iB e^{aw} \varphi dx \leq C \exp\left(\dfrac{a^2}{16\pi}\iB |\nabla w|^2 \varphi dx\right).\label{eq:6.3.1}
\end{align}

\end{lemma}

According to Lemma \ref{lem:6.2} and Lemma \ref{lem:6.3}, we will show the next lemma.
\begin{lemma}
Let $0 < r \ll 1$ and $\varphi = \varphi_{(0,r,n)}$ be as in Lemma \ref{lem:30}. Then it holds that
\begin{align}
\limsup_{t \to \infty} \iB |\nabla w|^2 \varphi dx = + \infty.\label{eq:6.4.1}
\end{align}
\end{lemma}

\begin{proof}
As a result of \eqref{eq:6.2.1}, we have for $ t \in (0,\infty)$,
\begin{align*}
\iB u\log u\varphi dx \leq \iB uw \varphi dx + C,
\end{align*}
where $C$ is a positive constant. Since the origin is the grow-up point, Proposition \ref{prop:3.11} implies that
\begin{align*}
\limsup_{t \to \infty} \iB uw \varphi = + \infty.
\end{align*}
By the generalized Young inequality: let $x > 0$ and any $y \in \R$, then it holds that
\begin{align*}
xy \leq x \log x + e^{y-1},
\end{align*}
we get for any $a > 0$,
\begin{align*}
a \iB uw \varphi dx &\leq \dfrac{1}{e}\iB e^{aw} \varphi dx + \iB (u\log u)\varphi dx\\
&\leq \iB e^{aw} \varphi dx + \iB (u\log u)\varphi dx\\
&\leq \F_\varphi(t) + \iB uw\varphi dx + \iB e^{aw} \varphi dx\\
&\leq C + \iB uw \varphi dx + \iB e^{aw} \varphi dx
\end{align*}
due to the positivity of $u$ and \eqref{eq:6.2.1}. Taking $a > 1$, we obtain that
\begin{align*}
\limsup_{t \to \infty} \iB e^{aw} \varphi dx = + \infty
\end{align*}
and thus 
\begin{align*}
\limsup_{t \to \infty} \iB |\nabla w|^2 \varphi dx = + \infty
\end{align*}
as a consequence of \eqref{eq:6.3.1}.
\end{proof}

The proof of the following lemma is based on the Jensen inequality and the convexity of $-\log s$ for $s > 0$.

\begin{lemma}[{\cite[Lemma 5.4]{NSS2000}}]
Let $0 < r \ll 1$ and $\varphi = \varphi_{(0,r,n)}$ be as in Lemma \ref{lem:30}. Then it holds that for any $ t \in (0,\infty)$,
\begin{align}
\iB uw \varphi dx \leq \iB (u\log u)\varphi dx + M_\varphi(t) \log \left(\iB e^w \varphi dx\right)
- M_\varphi(t)\log M_\varphi(t),\label{eq:6.5.1}
\end{align}
where $M_\varphi(t) := \iB u(x,t)\varphi(x) dx$.
\end{lemma}

Proposition \ref{prop:5.1} in the previous section provides information about a behavior of the solution $u$. Here, we will improve Proposition \ref{prop:5.1} as a result of the findings in this section. 

\begin{proposition}\label{prop:6.6}
Let $0 < r \ll 1$ and $\varphi = \varphi_{(0,r,n)}$ be as in Lemma \ref{lem:30}. Then  it holds that
\begin{align}
\limsup_{t \to \infty} \iB u\varphi dx \ge 8\pi.\label{eq:6.6}
\end{align}
\end{proposition}

\begin{proof}
Let us show the desired inequality. 
Suppose the contrary that
\begin{align}
\limsup_{t \to \infty} M_\varphi(t) = \limsup_{t \to \infty} \iB u\varphi dx < 8\pi.\label{eq:6.6.1}
\end{align}
By virtue of \eqref{eq:6.2.1}, \eqref{eq:6.3.1}, and \eqref{eq:6.5.1}, we get
\begin{align*}
&\dfrac{1}{2}\iB (|w_t|^2 + |\nabla w|^2 + w^2)\varphi dx\\
&= \F_\varphi (t) - \iB (u\log u)\varphi dx + \iB uw \varphi dx\\
&\leq M_\varphi(t) \log \left(\iB e^w \varphi dx\right) - M_\varphi (t) \log M_\varphi(t) + C\\
&\leq \dfrac{M_\varphi(t)}{16\pi}\iB |\nabla w|^2 \varphi dx + M_\varphi(t) \log C - M_\varphi(t)\log M_\varphi(t) + C
\end{align*}
and hence
\begin{align*}
\dfrac{1}{2}\left(1 - \dfrac{M_\varphi(t)}{8\pi}\right)\iB |\nabla w|^2 \varphi dx \leq M_\varphi(t) \log C + C.
\end{align*}
The assertion \eqref{eq:6.6.1} yields
\begin{align*}
\limsup_{t \to \infty}\iB |\nabla w|^2 \varphi dx \leq C,
\end{align*}
where $C > 0$ is a constant, which is contrary to \eqref{eq:6.4.1}. This completes the proof. 
\end{proof}

Finally we will give the proof of Theorem \ref{th:1}, building on the results established thus far. 
Our argument in this section is gained by reconstructing the method presented by Senba and Suzuki \cite{SS2001}, and Nagai, Senba, and Suzuki \cite{NSS2000}.
Here, although Theorem \ref{th:1} is similar to result achieved by \cite{SS2001, NSS2000},
the point to note that differ from ours is the existence of the limit of $u$ in the domain except the origin. If the maximal existence time $T_{\mathrm{max}}$ of the solutions is finite, the fact
\begin{align*}
\sup_{(x,t) \in ((\overline{\Omega}\setminus B_r) \times (0,T_{\mathrm{max}}))}|u_t(x,t)| < \infty
\end{align*}
makes it obvious that $\lim_{t \to T_{\mathrm{max}}} u(x,t)$ exists in $\overline{\Omega} \setminus B_r$ for any $r \in (0,R)$. However, we may not rule out the possibility of oscillation because of the maximal existence time $T_{\mathrm{max}} = \infty$. 
Hence we must carefully observe a large-time behavior of the solution to \eqref{p}, so that we make use of the compact theorem of the H\"older space instead of the argument above in \cite{SS2001, NSS2000}.

\begin{proof}[Proof of Theorem \ref{th:1}]
Let $0 < r \ll 1$ and $\varphi_{(0,r,n)}$ be as described in Lemma \ref{lem:30}. 
Here, from the Banach--Alaoglu theorem,
we pay attention to be able to take a subsequence $\{t_k\}$ of time with $t_k \to \infty$ as $k \to \infty$ such that for all $0 < r < 1$,
\begin{align*}
\lim_{k \to \infty} \iB u(x,t_k)\varphi_{(0,r,n)}(x) dx \ge 8\pi
\end{align*}
and we denote $m(r)$ by the left-hand side of the above inequality.
Since it follows from Lemma \ref{lem:30} that $m(r) - m(\frac{r}{2}) \ge 0$, $m(\cdot)$ is monotone decreasing for $0 < r \ll 1$ and bounded from below. On account of the above inequality, we see that there exists $m$ such that
\begin{align}
m = \lim_{r \to 0} m(r) \ge 8\pi.\label{eq:6.10.0}
\end{align}
Lemma \ref{lem:6.1} now gives for each $0 < r \ll 1$,
\begin{align*}
\sup_{t > 0}\|u(t)\|_{C^{2 + \alpha}(\overline{\Omega} \setminus B_r)} < \infty.
\end{align*}
For the reason that $\{u(\cdot,t_k)\}_{k \in \N}$ is relatively compact in $C^2(\overline{\Omega} \setminus B_r)$ with the help of the Arzel\`a--Ascoli theorem, the diagonal argument suggests there exist a subsequence of $\{t_k\}$, which is denoted by $\{t_k\}$, and the limit function $f \in C(\overline{\Omega} \setminus \{0\})$ such that
\begin{align}
\lim_{k \to \infty}u(x,t_k) = f(x)\ \mathrm{in}\ C_{loc}^2(\overline{\Omega} \setminus \{0\}).\label{eq:6.10.1}
\end{align}
According to the positivity of $u$ and the mass conservation law \eqref{pro:3.1.1}, we can deduce $f \ge 0$ and $f \in L^1(\Omega)$. What is left is to show the convergence to a delta function of the grow-up solution of \eqref{p}. 
To complete the proof, we calculate for $\xi \in C(\overline{\Omega})$ as follows:
\begin{align*}
&\iB u(x,t_k) \xi(x) dx - m \xi(0) - \iB f(x) \xi(x) dx\\
&= \xi(0)\left(\iB u(x,t_k) \varphi_{(0,r,n)}(x) dx - m\right) + \iB (\xi(x) - \xi(0)) u(x,t_k) \varphi_{(0,r,n)}(x) dx\\
&\hspace{0.5cm}- \iB \xi(x) f(x) \varphi_{(0,r,n)}(x) dx + \iB \xi(x)(u(x,t_k) - f(x))(1 - \varphi_{(0,r,n)}(x)) dx.
\end{align*}
Lemma \ref{lem:30} and \eqref{pro:3.1.1} lead to
\begin{align*}
&\left|\iB u(x,t_k) \xi(x) dx - m \xi(0) - \iB f(x) \xi(x) dx\right|\\
&\leq \|\xi\|_{L^\infty(\Omega)}\left|\iB u(x,t_k) \varphi_{(0,r,n)}(x) dx - m\right| + \|\xi - \xi(0)\|_{L^\infty(B_{2r})}\|u_0\|_{L^1(\Omega)}\\
&\hspace{0.5cm}+ \|\xi\|_{L^\infty(\Omega)}\iB f(x) \varphi_{(0,r,n)}(x) dx + \|\xi\|_{L^\infty(\Omega)}\|u(t_k)-f\|_{L^\infty(\Omega \setminus B_r)}
\end{align*}
and thanks to the uniform convergence \eqref{eq:6.10.1} mentioned above, it holds that
\begin{align*}
&\limsup_{k \to \infty}\left|\iB u(x,t_k) \xi(x) dx - m \xi(0) - \iB f(x) \xi(x) dx\right|\\
&\leq \|\xi\|_{L^\infty(\Omega)} |m(r) - m| + \|\xi - \xi(0)\|_{L^\infty(B_{2r})}\|u_0\|_{L^1(\Omega)} + \|\xi\|_{L^\infty(\Omega)}\iB f(x) \varphi_{(0,r,n)}(x)dx.
\end{align*}
Finally letting $r \to 0$, we can conclude from \eqref{eq:6.10.0} and $f \in L^1(\Omega)$ that
\begin{align*}
\limsup_{k \to \infty}\left|\iB u(x,t_k) \xi(x) dx - m \xi(0) - \iB f(x) \xi(x) dx\right| \leq 0,
\end{align*}
which is the desired conclusion.
\end{proof}

\subsection{Proof of Theorem \ref{th:3}}\label{subsec:3.6}

This section extends Theorem \ref{th:1} to derive new results under the assumption that the grow-up solution of \eqref{p} satisfies the uniform boundedness for $t\in [0,\infty )$ of the Lyapunov functional. We highlight that Theorem \ref{th:3} provides the interesting result unlike for the Keller--Segel system. Mizoguchi \cite{M2020_1, M2020_2} in the study of the Keller--Segel system showed the Lyapunov functional is uniformly bounded from below in time considering global-in-time solutions of the Keller--Segel system, and moreover the concentration of the mass in infinite time cannot occur. Therefore, we cannot give full play to the results in \cite{M2020_1}, nonetheless
the proof of Theorem \ref{th:3} is inspired by the argument using stationary solutions in \cite{M2020_1}. We first build a stationary solution to \eqref{p} which grows up at the origin in the same way as in \cite{WM2010}.
Throughout this subsection, we will continue to write the same symbol $t_k$ to denote a subsequence of time for abbreviation.

\begin{proposition}\label{pro:7.1}
Let $(u,v,w)$ be a radially symmetric grow-up solution of \eqref{p} with \eqref{c} and $\{t_k\}$ be a sequence of time with $t_k \to \infty$ as $k \to \infty$. If the solution $(u,v,w)$ satisfies
\begin{align*}
\inf_{t\ge 0} \F(t) > -\infty,
\end{align*}
then there exist a subsequence of time, which is denoted by $\{t_k\}$, and a radially symmetric function $(u_\infty, v_\infty, w_\infty) \in C^2(\overline{\Omega} \setminus \{0\}) \times C(\overline{\Omega} \setminus \{0\}) \times C^2(\overline{\Omega} \setminus \{0\})$ such that 
\begin{align}
&\lim_{k \to \infty} \iB u|\nabla (\log u- w)|^2(x,t_k) dx = 0,\label{eq:7.14.1}\\
&\lim_{k \to \infty} \iB |w_t(x,t_k)|^2 dx = 0,\label{eq:7.14.2}\\
&\lim_{k \to  \infty}(u(t_k),w(t_k)) = (u_\infty, w_\infty)\quad \mathrm{in}\ C_{loc}^2(\overline{\Omega} \setminus \{0\}),\label{eq:7.14.3}\\
&\lim_{k \to \infty}v(t_k) = v_\infty\quad \mathrm{in}\ C_{loc}(\overline{\Omega} \setminus \{0\}),\label{eq:7.14.4}
\end{align}
and moreover if $u_\infty \not\equiv 0$ in $\overline{\Omega} \setminus \{0\}$, then $(u_\infty, v_\infty, w_\infty)$ satisfies $u_\infty, v_\infty, w_\infty > 0$ in $\overline{\Omega} \setminus \{0\}$ as well as
\begin{equation}\label{sp}
\begin{cases}
u_\infty = v_\infty \qquad &\mathrm{in}\ \Omega \setminus \{0\},\\
-\Delta w_\infty + w_\infty  = u_\infty\qquad &\mathrm{in}\ \Omega \setminus \{0\},\\
\nabla (\log u_\infty - w_\infty) = 0\qquad &\mathrm{in}\ \Omega \setminus \{0\},\\
\dfrac{\partial u_\infty}{\partial \nu} = \dfrac{\partial w_\infty}{\partial \nu} = 0 \qquad &\mathrm{on}\ \partial \Omega.
\end{cases}
\end{equation}
\end{proposition}

Let us recall Proposition \ref{prop:6.6}: for $0 < r \ll 1$ and $\varphi = \varphi_{(0,r,n)}$ as in Lemma \ref{lem:30}, it holds that
\begin{align*}
\limsup_{ t \to \infty}\iB u(x,t) \varphi(x) dx \ge 8\pi.
\end{align*}
Taking a subsequence of time $\{t_k\}$ by the definition of $\limsup$, hereafter we will apply the above proposition with such a subsequence $\{t_k\}$.

\begin{lemma}\label{lem:7.11}
Let $u_\infty$ be as in Proposition \ref{pro:7.1} and $\|u_0\|_{L^1(\Omega)}> 8\pi$. Then it holds that
\begin{align}
\iB u_\infty (x) dx \leq \|u_0\|_{L^1(\Omega)} - 8\pi.\label{eq:7.2.1}
\end{align}
\end{lemma}

\begin{proof}
Let $0 < r \ll 1$ and $\varphi = \varphi_{(0,r,n)}$ be as in Lemma \ref{lem:30}, then Proposition \ref{prop:3.1.0} and Proposition \ref{prop:6.6} imply
\begin{align*}
\|u_0\|_{L^1(\Omega)} &= \lim_{k \to \infty} \iB u(x,t_k) \varphi (x) dx + \lim_{k \to \infty} \iB u(x,t_k)(1-\varphi (x)) dx\\
&\ge 8\pi + \iB u_\infty(x)(1- \varphi (x)) dx
\end{align*}
owing to \eqref{eq:7.14.3}. Then Fatou's lemma allows us to complete the proof.
\end{proof}

The following lemma is obtained by modifying \cite[Lemma 4.3]{M2020_1}.
\begin{lemma}\label{lem:7.12}
Let $w_\infty$ be as in Proposition \ref{pro:7.1}. Then it holds that
\begin{align}
\iB w_\infty (x) dx = \|u_0\|_{L^1(\Omega)}.\label{eq:7.3.1}
\end{align}
\end{lemma}

\begin{proof}
The third equation of \eqref{p} and the $0$-Neumann boundary conditions infer that
\begin{align*}
\iB w_t(x,t_k) dx = -\iB w(x,t_k) dx + \iB v(x,t_k) dx.
\end{align*}
The second equation of \eqref{p} allows us to estimate as follows:
\begin{align*}
\iB w_t(x,t_k) dx = -\iB w(x,t_k) dx + e^{-t_k} \|v_0\|_{L^1(\Omega)} +  \int_0^{t_k} e^{-(t_k -s)}\|u(s)\|_{L^1(\Omega)} ds
\end{align*}
and according to the mass conservation law \eqref{pro:3.1.1}, we have
\begin{align*}
\iB w_t(x,t_k) dx = -\iB w(x,t_k) dx + e^{-t_k}\|v_0\|_{L^1(\Omega)} + (1-e^{-t_k})\|u_0\|_{L^1(\Omega)}.
\end{align*}
As to the left-hand side, using the H\"{o}lder inequality to obtain
\begin{align*}
\left|\iB w_t(x,t_k) dx\right| &\leq \iB |w_t(x,t_k)| dx\\
&\leq \left(|\Omega|\iB |w_t(x,t_k)|^2 dx\right)^\frac{1}{2},
\end{align*}
we get from \eqref{eq:7.14.2} that
\begin{align*}
\lim_{k \to \infty} \iB w_t(x,t_k) dx = 0.
\end{align*}
At the same time, since Proposition \ref{pro:322} makes the improper integral possible, we deduce from \eqref{eq:7.14.3} and Lebesgue's dominated convergence theorem that
\begin{align*}
\lim_{k \to \infty} \iB w(x,t_k) dx = \iB w_\infty (x) dx,
\end{align*}
and this clearly implies our claim.
\end{proof}

We will give the following lemma introduced by Mizoguchi \cite[Lemma 4.4]{M2020_1}, which completes the proof of Theorem \ref{th:3}.

\begin{lemma}\label{lem:7.13}
Let $m_0 \in (1,4)$ and $(u_\infty, v_\infty, w_\infty)$ be as in Proposition \ref{pro:7.1}. If $u_\infty \not\equiv 0$ in $\overline{\Omega} \setminus \{0\}$, then it holds that
\begin{align*}
\iB u_\infty (x) dx - \iB w_\infty (x) dx \ge -2\pi m_0.
\end{align*}
\end{lemma}

\begin{proof}[Proof of Theorem \ref{th:3}]
Suppose $u_\infty \not\equiv 0$ in $\overline{\Omega} \setminus \{0\}$. Thanks to Lemma \ref{lem:7.11}-\ref{lem:7.13}, we have
\begin{align*}
-2\pi m_0 + \|u_0\|_{L^1(\Omega)} &= -2\pi m_0 + \iB w_\infty (x) dx\\
&\leq \iB u_\infty (x) dx\\
&\leq \|u_0\|_{L^1(\Omega)} - 8\pi.
\end{align*}
Since $m_0 \in (1,4)$ and $\|u_0\|_{L^1(\Omega)} > 8\pi$, this is impossible. Therefore we obtain $u_\infty \equiv 0$ in $\overline{\Omega} \setminus \{0\}$. Let $0 < r \ll 1$ and $\varphi = \varphi_{(0,r,n)}$ be as in Lemma \ref{lem:30}, we can sharpen Proposition \ref{prop:6.6} and prove that
\begin{align}
\|u_0\|_{L^1(\Omega)} &= \lim_{k \to \infty} \iB u(x, t_k)\varphi (x) dx + \lim_{k \to \infty} \iB u(x, t_k)(1-\varphi (x)) dx\notag\\
&= \lim_{k \to \infty} \iB u(x,t_k)\varphi(x) dx\label{eq:7.222}
\end{align}
owing to the mass conservation law \eqref{pro:3.1.1} and \eqref{eq:7.14.3}. We are now ready to finish the proof. The analysis similar to that in the proof of Theorem \ref{th:1} shows for $\xi \in C(\overline{\Omega})$,
\begin{align*}
&\iB u(x,t_k)\xi (x) dx - \|u_0\|_{L^1(\Omega)}\xi (0)\\
&= \iB u(x,t_k) \xi(x) dx - \iB u(x,t_k) \xi(0) dx\\
&= \iB u(x,t_k)(\xi(x) - \xi(0)) dx\\
&=\iB u(x,t_k)\varphi(x)(\xi(x)-\xi(0)) dx + \iB u(x,t_k)(1-\varphi(x))(\xi(x)-\xi(0))dx
\end{align*}
due to \eqref{pro:3.1.1}. Lemma \ref{lem:30} yields that
\begin{align*}
&\left|\iB u(x,t_k)\xi (x) dx - \|u_0\|_{L^1(\Omega)}\xi (0)\right|\\
&\leq \|\xi - \xi(0)\|_{L^\infty(B_{2r})}\iB u(x,t_k)\varphi(x) dx + \|\xi - \xi(0)\|_{L^\infty(\Omega \setminus B_r)} \iB u(x,t_k)(1-\varphi(x)) dx.
\end{align*}
Hence \eqref{eq:7.14.3} and \eqref{eq:7.222} contribute to achieving that 
\begin{align*}
\lim_{k \to \infty} \left|\iB u(x,t_k)\xi (x) dx - \|u_0\|_{L^1(\Omega)}\xi (0)\right| \leq 
\|\xi - \xi(0)\|_{L^\infty(B_{2r})}\|u_0\|_{L^1(\Omega)}
\end{align*}
and consequently, this implies the end of the proof by letting $r \to 0$.
\end{proof}

\textbf{Acknowledgments}

The author would like to thank the referee for careful reading of the
manuscript and for constructive comments.
The author wishes to thank my supervisor Kentaro Fujie in Tohoku University for several helpful comments and his encouragement concerning this paper.

\textbf{Data availability}

Data sharing is not applicable to this article since no datasets were generated or analyzed in this study.

\textbf{Declarations}

\textbf{Conflict of interest}

The author declares that there is no conflict of interest.


\bigskip
\address{ 
Mathematical Institute \\
Tohoku University \\
Sendai 980-8578 \\
Japan
}
{soga.yuri.q6@dc.tohoku.ac.jp}

\end{document}